\documentclass{amsart}
\usepackage{amssymb,amsmath,amsthm}
\newtheorem{theorem}{Theorem}[section]
\newtheorem{lemma}[theorem]{Lemma}
\newtheorem{remark}{Remark}[section]
\usepackage{esint}
\usepackage{xcolor}

\usepackage{hyperref}
\hypersetup{
	colorlinks = true,%
	citecolor = blue,%[rgb]{0.65,0.0,0.0},
	filecolor=red,%
	linkcolor = [rgb]{0.65,0.0,0.0},%
	anchorcolor = red,
	pagecolor = red,
	urlcolor= [rgb]{0.65,0.0,0.0}
	linktocpage=true,
	pdfpagelabels=true,
	bookmarksnumbered=true,
}

\numberwithin{equation}{section}

\usepackage{color}

\begin{document}

\title{Finite many critical problems involving Fractional
Laplacians in $\mathbb{R}^{N}$}

%    Remove any unused author tags.

%    author one information

%    author two information

\author{Yu Su}
\address{School of Mathematics and Statistics, Central South University, Changsha, 410083 Hunan, P.R.China.}
\curraddr{}
\email{yizai52@qq.com}
\thanks{}

\author{Haibo Chen}
\address{School of Mathematics and Statistics, Central South University, Changsha, 410083 Hunan, P.R.China.}
\curraddr{}
\email{math\_chb@163.com}
\thanks{This research was supported by National Natural Science Foundation of China 11671403.}



\date{}

\dedicatory{}

\begin{abstract}
In this paper,
we consider the nonlocal elliptic problems in
$\mathbb{R}^{N}$,
which involve  finite many critical exponents.
By using endpoint refined Hardy--Sobolev inequality,
fractional Coulomb--Sobolev space and variational method,
we
establish the existence of nonnegative solution.
Our results generalize some results obtained by Yang and Wu [Adv. Nonlinear Stud. (2017) \cite{Yang2017}].
\par
\end{abstract}

\maketitle

\section{Introduction}
In this paper,
we  consider the following problems:
$$
(-\Delta)^{s} u
=
\sum_{i=1}^{k}
\left(
\int_{\mathbb{R}^{N}}
\frac{|u|^{2^{*}_{\alpha_{i}}}}{|x-y|^{\alpha_{i}}}
\mathrm{d}y
\right)
|u|^{2^{*}_{\alpha_{i}}-2}u
+
|u|^{2^{*}_{s}-2}u
,
\mathrm{~in~}
\mathbb{R}^{N},
\eqno(\mathcal{P}_{1})
$$
and
$$
(-\Delta)^{s} u
-\frac{\zeta u}{|x|^{2s}}
=
\left(
\int_{\mathbb{R}^{N}}
\frac{|u|^{2^{*}_{\alpha}}}{|x-y|^{\alpha}}
\mathrm{d}y
\right)
|u|^{2^{*}_{\alpha}-2}u
+
\sum_{i=1}^{k}
\frac{|u|^{2^{*}_{s,\theta_{i}}-2}u}
{|x|^{\theta_{i}}}
,
\mathrm{~in~}
\mathbb{R}^{N},
\eqno(\mathcal{P}_{2})
$$
where
$N\geqslant3$,
$s\in(0,1)$,
$\zeta\in
\left[
0,4^{s}\frac{\Gamma(\frac{N+2s}{4})}{\Gamma(\frac{N-2s}{4})}
\right)$,
$\alpha\in(0,N)$,
$2^{*}_{s}=\frac{2N}{N-2s}$
is the critical Sobolev  exponent,
$2^{*}_{s,\theta_{i}}=\frac{2(N-\theta_{i})}{N-2s}$
are the critical Hardy--Sobolev  exponents,
$2^{*}_{\alpha_{i}}=
\frac{2N-\alpha_{i}}{N-2s}$
are the Hardy--Littlewood--Sobolev upper critical exponents,
the parameters
$\alpha_{i}$
and
$\theta_{i}$
satisfy the assumptions:

\noindent
($H_{1}$)
$0<\alpha_{1}<\alpha_{2}<\cdots<\alpha_{k}<N$
($k\in \mathbb{N}$,
$2\leqslant k<\infty$);

\noindent
($H_{2}$)
$0<\theta_{1}<\cdots<\theta_{k}<2s$
($k\in \mathbb{N},~2\leqslant k<\infty$),
and
$2\theta_{k}-\theta_{1}\in(0,2s)$.

The fractional Laplacian $(-\Delta)^{s}$
of a function
$u: \mathbb{R}^{N}\rightarrow\mathbb{R}$
can be defined as
$$(-\Delta)^{s}u=\mathcal{F}^{-1}(|\xi|^{2s}\mathcal{F}(u)(\xi)),~\mathrm{for~all}~\xi\in \mathbb{R}^{N},$$
and for $u\in C^{\infty}_{0}(\mathbb{R}^{N})$,
where $\mathcal{F}(u)$ denotes the Fourier transform of $u$.
The operator
$(-\Delta)^{s}$
in
$\mathbb{R}^{N}$
is a nonlocal pseudo--differential operator taking the form
\begin{equation*}
\begin{aligned}
(-\Delta)^{s}u(x)
=
C_{N,s}\mathrm{P. V.}
\int_{\mathbb{R}^{N}}
\frac{u(x)-u(y)}{|x-y|^{N+2s}}
\mathrm{d}y,
\end{aligned}
\end{equation*}
where
$\mathrm{P. V.}$
is the Cauchy principal value and
$C_{N,s}$
is a normalization constant.
The fractional
power of Laplacian is the infinitesimal generator of L\'{e}vy stable diffusion process and arise
in anomalous diffusion in plasma, population dynamics, geophysical fluid dynamics, flames
propagation, minimal surfaces and game theory
(see
\cite{Applebaum2004,Caffarelli2012,Garroni2005}).

Problem
$(\mathcal{P}_{1})$
and
$(\mathcal{P}_{2})$
are related to the nonlinear Choquard equation as follows:
\begin{equation}\label{1}
-\Delta u
+
V(x)u
=
\left(
|x|^{\alpha}*|u|^{q}
\right)
|u|^{q-2}u,
\mathrm{~in~}
\mathbb{R}^{N},
\end{equation}
where
$\frac{2N-\alpha}{N}\leqslant q\leqslant\frac{2N-\alpha}{N-2}$
and
$\alpha\in(0,N)$.
For
$q=2$
and
$\alpha=1$,
the problem
(\ref{1})
goes back to the description of the quantum theory of a polaron at rest by Pekar in 1954
\cite{Pekar1954}
and the modeling of an electron trapped in its own hole in
1976
in the work of Choquard,
as a certain approximation to Hartree--Fock theory of one--component plasma
\cite{Penrose1996}.
The existence and qualitative properties of solutions of Choquard type equations
(\ref{1})
have been widely studied in the last decades
(see \cite{Moroz2016}).

For Laplacian with nonlocal Hartree type nonlinearities,
Gao and Yang \cite{Gao2016} investigated the following critical Choquard equation:
\begin{equation}\label{2}
\begin{aligned}
-\Delta u
=
\left(
\int_{\mathbb{R}^{N}}
\frac{|u|^{2^{*}_{h,\alpha}}}{|x-y|^{\alpha}}
\mathrm{d}y
\right)
|u|^{2^{*}_{h,\alpha}-2}u
+\lambda u,
\mathrm{~in~}
\Omega,
\end{aligned}
\end{equation}
where
$\Omega$
is a bounded domain of
$\mathbb{R}^{N}$,
with lipschitz boundary,
$N\geqslant3$,
$\alpha\in(0,N)$
and
$\lambda>0$.
By using variational methods,
they established the existence, multiplicity and nonexistence of nontrivial solutions to equation (\ref{2}).
For details and recent works we refer to \cite{O.Alves2017,Gao2017JMAA,Mercuri2016} and the references therein.

For fractional Laplacian with nonlocal Hartree--type nonlinearities,
D'Avenia, Siciliano and Squassina
\cite{d'Avenia2015}
considered the following fractional Choquard equation:
\begin{equation}\label{3}
\begin{aligned}
(-\Delta)^{s} u
+
\omega u
=
\left(
\mathcal{K}_{\alpha}\ast|u|^{q}
\right)
|u|^{q-2}u,
\mathrm{~in~}
\mathbb{R}^{N},
\end{aligned}
\end{equation}
where
$N\geqslant3$,
$s\in(0,1)$,
$\omega\geqslant0$,
$\alpha\in(0,N)$
and
$q\in(\frac{2N-\alpha}{N},\frac{2N-\alpha}{N-2s})$.
In particularly,
when
$\omega=0$,
$\alpha=4s$
and
$q=2$,
then peoblem
(\ref{3})
become a fractional Choquard euqation with upper critical exponent in the sense of Hardy--Littlewood--Sobolev inequatlity as follows:
\begin{equation}\label{4}
\begin{aligned}
(-\Delta)^{s} u
=
\left(
\int_{\mathbb{R}^{N}}
\frac{|u|^{2}}{|x-y|^{4s}}
\mathrm{d}y
\right)
u,
\mathrm{~in~}
\mathbb{R}^{N}.
\end{aligned}
\end{equation}
D'Avenia, Siciliano and Squassina in
\cite{d'Avenia2015}
obtained regularity, existence, nonexistence of nontrivial solutions to problem (\ref{3}) and problem (\ref{4}).
Mukherjee and Sreenadh \cite{Mukherjee2017Fractional}
extended the study of problem \eqref{2} to fractional Laplacian equation.

Recently,
Yang and Wu \cite{Yang2017}
studied the following nonlocal elliptic problems:
\begin{equation}\label{7}
\begin{aligned}
(-\Delta)^{s} u
-\frac{\zeta u}{|x|^{2s}}
=
\left(
\int_{\mathbb{R}^{N}}
\frac{|u|^{2^{*}_{\alpha}}}{|x-y|^{\alpha}}
\mathrm{d}y
\right)
|u|^{2^{*}_{\alpha}-2}u
+
\left(
\int_{\mathbb{R}^{N}}
\frac{|u|^{2^{*}_{\beta}}}{|x-y|^{\beta}}
\mathrm{d}y
\right)
|u|^{2^{*}_{\beta}-2}u
,
\mathrm{~in~}
\mathbb{R}^{N},
\end{aligned}
\end{equation}
and
\begin{equation}\label{8}
\begin{aligned}
(-\Delta)^{s} u
-\frac{\zeta u}{|x|^{2s}}
=
\left(
\int_{\mathbb{R}^{N}}
\frac{|u|^{2^{*}_{\alpha}}}{|x-y|^{\alpha}}
\mathrm{d}y
\right)
|u|^{2^{*}_{\alpha}-2}u
+
\frac{|u|^{2^{*}_{s,\theta}-2}u}{|x|^{\theta}},
\mathrm{~in~}
\mathbb{R}^{N},
\end{aligned}
\end{equation}
where
$N\geqslant3$,
$s\in(0,1)$,
$\zeta\in
\left[
0,4^{s}\frac{\Gamma(\frac{N+2s}{4})}{\Gamma(\frac{N-2s}{4})}
\right)$,
$\alpha,\beta\in(N-2s,N)$,
$\theta\in(0,2s)$,
$2^{*}_{\alpha}=\frac{2N-\alpha}{N-2s}$
and
$2^{*}_{s,\theta}=\frac{2(N-\theta)}{N-2s}$.
Using the refinement of the Sobolev inequality which is related to
the Morrey space,
they showed the existence of nontrivial solutions for problem
(\ref{7})
and
problem
(\ref{8}).
In \cite{ZhangBL2017},
Wang, Zhang and Zhang extended the study of problem (\ref{8}) to the fractional Laplacian system.
By using variational methods,
they investigated the extremals of the corresponding best fractional Hardy--Sobolev constant and established the existence of solutions to the fractional Laplacian system.

Moreover,
there are many other kinds of problem involving two critical nonlinearities,
such as the Laplacian
$-\Delta$
(see \cite{Li2012,Seok2018,Zhong2016}),
the p--Laplacian
$-\Delta_{p}$
(see \cite{Pucci2009}),
the biharmonic operator
$\Delta^{2}$ (see \cite{Bhakta2015}),
and the fractional operator
$(-\Delta)^{s}$
(see \cite{Ghoussoub2016,Chen2018}).

There are two questions arise:

{\bf Question 1:
For
$\zeta=0$,
can we extend the study of problem
(\ref{7})
in the finite many critical nonlinearities?}

{\bf Question 2:
Can we extend the studies of problem
(\ref{8})
and
problem
(\ref{3})
in the finite many critical nonlinearities?}

We answer above questions in this paper.
To our knowledge, there are no results in these senses.

The variational approach that we adopt here, relies on the following
inequalities:
\begin{lemma}\label{lemma1}
$\left.\right.$\cite[Hardy-Littlewood-Sobolev~inequality]{Lieb2001}
Let
$t,r>1$
and
$0<\alpha<N$
with
$\frac{1}{t}+\frac{1}{r}+\frac{\alpha}{N}=2$,
$f\in L^{t}(\mathbb{R}^{N})$
and
$h\in L^{r}(\mathbb{R}^{N})$.
There exists a sharp constant
$C(N,\alpha,t,r)>0$,
independent of
$f,g$
such that
$$\int_{\mathbb{R}^{N}}\int_{\mathbb{R}^{N}}
\frac{|f(x)||h(y)|}
{|x-y|^{\alpha}}
\mathrm{d}x\mathrm{d}y
\leqslant
C(N,\alpha,t,r)
\|f\|_{t}
\|h\|_{r}.$$
If
$t=r=\frac{2N}{2N-\alpha}$,
then
$$C(N,\alpha,t,r)
=
C(N,\alpha)
=
\pi^{\frac{\alpha}{2}}
\frac{\Gamma(\frac{N}{2}-\frac{\alpha}{2})}{\Gamma(N-\frac{\alpha}{2})}
\left\{
\frac{\Gamma(\frac{N}{2})}{\Gamma(N)}
\right\}^{\frac{\alpha-N}{N}}.$$
\end{lemma}
\begin{lemma}\label{lemma2}
$\left.\right.$\cite[Endpoint refined Sobolev inequality]{Bellazzini2016}
Let
$s\in(0,\frac{N}{2})$
and
$\alpha\in(0,N)$.
Then there exists a  constant
$C_{1}>0$
such that the inequality
$$
\|u\|_{L^{2^{*}_{s}}(\mathbb{R}^{N})}
\leqslant
C_{1}
\|u\|_{D}^{\frac{(N-\alpha)(N-2s)}{N(N+2s-\alpha)}}
\left(
\int_{\mathbb{R}^{N}}
\int_{\mathbb{R}^{N}}
\frac{|u(x)|^{2^{*}_{\alpha}}|u(y)|^{2^{*}_{\alpha}}}{|x-y|^{\alpha}}
\mathrm{d}x
\mathrm{d}y
\right)
^{\frac{s(N-2s)}{N(N+2s-\alpha)}},$$
holds for all
$u\in \mathcal{E}^{s,\alpha,2^{*}_{\alpha}}(\mathbb{R}^{N})$.
\end{lemma}
In  particular,
the Coulomb--Sobolev space and endpoint refined Sobolev inequality play the key roles in this paper.
For
$s=1$,
Mercuri, Moroz and Schaftingen \cite{Mercuri2016} introduced the Coulomb--Sobolev space and a family of associated optimal interpolation inequalities (include endpoint refined Sobolev inequality).
They studied the existence of solutions of the nonlocal Schr\"{o}dinger--Poisson--Slater type equation by Coulomb--Sobolev space and endpoint refined Sobolev inequality.
For $s\not=1$,
Bellazzini, Ghimenti, Mercuri, Moroz and Schaftingen \cite{Bellazzini2016} studied the fractional Coulomb--Sobolev space and endpoint refined Sobolev inequality.

The first
result of this paper is as follows.
\begin{theorem}\label{theorem1}
Let
$N\geqslant3$,
$s\in(0,1)$
and
$( H_{1})$
hold.
Then
problem
$(\mathcal{P}_{1})$
has a nonnegative solution $\bar{v}(x)$.
Moreover,
set
\begin{equation*}
\begin{aligned}
\bar{\bar{v}}(x)
=
\frac{1}{|x|^{N-2s}}
\bar{v}
\left(
\frac{x}{|x|^{2}}
\right).
\end{aligned}
\end{equation*}
Then
$\bar{\bar{v}}(x)$
is a nonnegative solution of the problem
\begin{equation*}
\begin{aligned}
(-
\Delta)^{s}
\bar{\bar{v}}
=&
\sum\limits^{k}_{i=1}
\left(
\int_{\mathbb{R}^{N}}
\frac{|\bar{\bar{v}}|^{\frac{2N-\alpha_{i}}{N-2s}}}
{|x-y|^{\alpha_{i}}}
\mathrm{d}y
\right)
\left|
\bar{\bar{v}}
\right|^{\frac{4s-\alpha_{i}}{N-2s}}
\bar{\bar{v}}
+
\left|
\bar{\bar{v}}
\right|^{\frac{4s}{N-2s}}
\bar{\bar{v}}
,
~\mathrm{in}~\mathbb{R}^{N}\backslash\{0\}.
\end{aligned}
\end{equation*}
\end{theorem}
\begin{remark}
Problem
$(\mathcal{P}_{1})$ is invariant under the weighted dilation
$$u\mapsto \tau^{\frac{N-2s}{2}}u(\tau x).$$
Therefore, it is well known that the mountain pass theorem does not yield critical points,
but only the Palais--Smale sequences.
In this type of situation,
it is necessary to show the non--vanishing of Palais--Smale sequences.
There are finite many Hardy--Littlewood--Sobolev critical exponents in problem
$(\mathcal{P}_{1})$,
it is difficult to show the non--vanishing of Palais--Smale sequences.
By using fractional Coulomb--Sobolev space, endpoint refined Sobolev inequality and Lemma \ref{lemma6},
we overcome this difficult in Lemma \ref{lemma19}.
\end{remark}
The second
result of this paper is as follows.
\begin{theorem}\label{theorem2}
Let
$N\geqslant3$,
$\alpha\in(0,   N)$,
$s\in(0,1)$
and
$( H_{2})$
hold.
Then
problem
$(\mathcal{P}_{2})$
has a nonnegative solution $\tilde{u}(x)$.
Moreover,
set
\begin{equation*}
\begin{aligned}
\tilde{\tilde{u}}(x)
=
\frac{1}{|x|^{N-2s}}
\tilde{u}
\left(
\frac{x}{|x|^{2}}
\right).
\end{aligned}
\end{equation*}
Then
$\tilde{\tilde{u}}(x)$
is a nonnegative solution of the problem
\begin{equation*}
\begin{aligned}
(-
\Delta)^{s}
\tilde{\tilde{u}}
-\zeta
\frac{\bar{\bar{v}}}{|x|^{2s}}
=&
\left(
\int_{\mathbb{R}^{N}}
\frac{|\tilde{\tilde{u}}|^{\frac{2N-\alpha}{N-2s}}}
{|x-y|^{\alpha}}
\mathrm{d}y
\right)
\left|
\tilde{\tilde{u}}
\right|^{\frac{4s-\alpha}{N-2s}}
\tilde{\tilde{u}}
+
\sum\limits^{k}_{i=1}
\frac{
\left|
\tilde{\tilde{u}}
\right|^{\frac{4s-2\theta_{i}}{N-2s}}
\tilde{\tilde{u}}}
{|x|^{\theta_{i}}},
~\mathrm{in}~\mathbb{R}^{N}\backslash\{0\}.
\end{aligned}
\end{equation*}
\end{theorem}
\begin{remark}
This paper not only extends the studies of problem (\ref{7}) and problem (\ref{8}) in the finite many critical nonlinearities,
but also extends $\alpha\in(N-2s,N)$ to $\alpha\in(0,N)$.
In \cite{ZhangBL2017} and \cite{Yang2017},
the authors just studied the case of $\alpha\in(N-2s,N)$.
It is nature to ask the case of $\alpha\in(0,N-2s)$.
In order to overcome this difficult,
we show the refinement of Hardy--Littlewood--Sobolev inequality for the case of $\alpha\in(0,N)$ (see Lemma \ref{lemma5}),
and
the endpoint refined Hardy--Sobolev inequality  (see Lemma \ref{lemma7}).
\end{remark}
\section{Preliminaries}
The space
$H^{s}(\mathbb{R}^{N})$
is defined as
$$H^{s}(\mathbb{R}^{N})=\{u\in L^{2}(\mathbb{R}^{N})
|(-\Delta)^{\frac{s}{2}}u\in L^{2}(\mathbb{R}^{N})\}.$$
This space is endowed with the norm
$$\|u\|_{H}^{2}=\|(-\Delta)^{\frac{s}{2}}u\|_{2}^{2}+\|u\|_{2}^{2}.$$
The space
$D^{s,2}(\mathbb{R}^{N})$
is the completion of
$C^{\infty}_{0}(\mathbb{R}^{N})$
with respect to the norm
$$
\|u\|_{D}^{2}=\|(-\Delta)^{\frac{s}{2}}u\|_{2}^{2}.
$$
It is well known that
$\Lambda=4^{s}\frac{\Gamma^{2}(\frac{N+2s}{4})}{\Gamma^{2}(\frac{N-2s}{4})}$
is the best constant in the Hardy inequality
$$
\Lambda
\int_{\mathbb{R}^{N}}
\frac{ u^{2}}{|x|^{2s}}
\mathrm{d}x
\leqslant
\|u\|_{D}^{2}
,~~
\mathrm{for~any~}
u\in
D^{s,2}(\mathbb{R}^{N}).
$$
By Hardy inequality and $\zeta\in[0,\Lambda)$,
we derive that
$$
\|u\|_{\zeta}^{2}
=
\|u\|_{D}^{2}
-
\zeta
\int_{\mathbb{R}^{N}}
\frac{ |u|^{2}}{|x|^{2s}}
\mathrm{d}x,
$$
is an equivalent norm in
$D^{s,2}(\mathbb{R}^{N})$,
since the following inequalities hold:
$$\left(1-\frac{\zeta}{\Lambda}\right)\|u\|_{D}^{2}\leqslant\|u\|_{\zeta}^{2}\leqslant\|u\|_{D}^{2}.$$
For
$\alpha\in(0,N)$
and
$s\in(0,1)$,
the fractional Coulomb--Sobolev space \cite{Bellazzini2016} is defined by
$$
\mathcal{E}^{s,\alpha,2^{*}_{\alpha}}(\mathbb{R}^{N})=
\left\{
\|u\|_{D}<\infty
~\mathrm{and}~
\int_{\mathbb{R}^{N}}
\int_{\mathbb{R}^{N}}
\frac{|u(x)|^{2^{*}_{\alpha}}|u(y)|^{2^{*}_{\alpha}}}{|x-y|^{\alpha}}
\mathrm{d}x
\mathrm{d}y<\infty
\right\}.
$$
We endow the space
$\mathcal{E}^{s,\alpha,2^{*}_{\alpha}}(\mathbb{R}^{N})$
with the norm
$$
\|u\|_{\mathcal{E},\alpha}^{2}
=
\|u\|_{D}^{2}
+
\left(
\int_{\mathbb{R}^{N}}
\int_{\mathbb{R}^{N}}
\frac{|u(x)|^{2^{*}_{\alpha}}|u(y)|^{2^{*}_{\alpha}}}{|x-y|^{\alpha}}
\mathrm{d}x
\mathrm{d}y
\right)^{\frac{1}{2^{*}_{\alpha}}}.
$$
For
$\alpha\in[0,N)$,
$\zeta\in[0,\Lambda)$
and
$s\in(0,1)$,
we define the best constant:
\begin{equation}\label{10}
\begin{aligned}
S_{\zeta,\alpha}:=
\inf_{u\in D^{s,2}(\mathbb{R}^{N})\setminus\{0\}}
\frac{\|u\|_{D}^{2}
-
\zeta
\int_{\mathbb{R}^{N}}
\frac{ |u|^{2}}{|x|^{2s}}
\mathrm{d}x}
{\left(
\int_{\mathbb{R}^{N}}
\int_{\mathbb{R}^{N}}
\frac{|u(x)|^{2^{*}_{\alpha}}|u(y)|^{2^{*}_{\alpha}}}{|x-y|^{\alpha}}
\mathrm{d}x
\mathrm{d}y
\right)^{\frac{1}{2^{*}_{\alpha}}}}.
\end{aligned}
\end{equation}
We know that
$S_{\zeta,\alpha}$
is attained in
$\mathbb{R}^{N}$
(see \cite{Yang2017}).
For
$s\in(0,1)$
and
$\theta\in(0,2s)$,
we define the best constant:
\begin{equation}\label{11}
\begin{aligned}
H_{\zeta,\theta}:=
\inf_{u\in D^{s,2}(\mathbb{R}^{N})\setminus\{0\}}
\frac{\|u\|_{D}^{2}
-
\zeta
\int_{\mathbb{R}^{N}}
\frac{ |u|^{2}}{|x|^{2s}}
\mathrm{d}x}
{\left(
\int_{\mathbb{R}^{N}}
\frac{|u|^{2^{*}_{s,\theta}}}{|x|^{\theta}}
\mathrm{d}x
\right)^{\frac{2}{2^{*}_{s,\theta}}}}
\end{aligned}
\end{equation}
where
$H_{\theta}$
is attained in
$\mathbb{R}^{N}$
(see \cite{Yang2015}).
A measurable function
$u:\mathbb{R}^{N}\rightarrow \mathbb{R}$
belongs to the Morrey space
$\|u\|_{\mathcal{L}^{p,\varpi}}(\mathbb{R}^{N})$
with
$p\in[1,\infty)$
and
$\varpi\in(0,N]$
if and only if
$$
\|u\|^{p}_{\mathcal{L}^{p,\varpi}(\mathbb{R}^{N})}
=
\sup_{R>0,x\in\mathbb{R}^{N}}
R^{\varpi-N}
\int_{B(x,R)}
|u(y)|^{p}
\mathrm{d}y
<\infty.
$$
\begin{lemma}
\label{lemma4}
$\left.\right.$
\cite[Theorem 1]{Palatucci2014}
For
$s\in (0,\frac{N}{2})$,
there exists
$C_{2}>0$
such that
for
$\iota$
and
$\vartheta$
satisfying
$\frac{2}{2^{*}_{s}}\leqslant\iota<1$,
$1\leqslant \vartheta<2^{*}_{s}=\frac{2N}{N-2s}$,
we have
\begin{align*}
\left(
\int_{\mathbb{R}^{N}}
|u|^{2^{*}_{s}}
\mathrm{d}x
\right)^{\frac{1}{2^{*}_{s}}}
\leqslant
C_{2}
\|u\|_{D}^{\iota}
\|u\|_{\mathcal{L}^{\vartheta,\frac{\vartheta(N-2s)}{2}}(\mathbb{R}^{N})}^{1-\iota},
\end{align*}
for any
$u\in D^{s,2}(\mathbb{R}^{N})$.
\end{lemma}
We introduce the energy functionals associated to  problems $(\mathcal{P}_{i})$ $(i=1,2,3)$ by
\begin{equation*}
\begin{aligned}
I_{1}(u)
=&
\frac{1}{2}
\|u\|_{D}^{2}
-
\sum_{i=1}^{k}
\frac{1}{2\cdot2^{*}_{\alpha_{i}}}
\int_{\mathbb{R}^{N}}
\int_{\mathbb{R}^{N}}
\frac{|u(x)|^{2^{*}_{\alpha_{i}}}|u(y)|^{2^{*}_{\alpha_{i}}}}{|x-y|^{\alpha_{i}}}
\mathrm{d}x
\mathrm{d}y
-
\frac{1}{2^{*}_{s}}
\int_{\mathbb{R}^{N}}
|u|^{2^{*}_{s}}
\mathrm{d}x,\\
I_{2}(u)
=&
\frac{1}{2}
\|u\|_{\zeta}^{2}
-
\frac{1}{2\cdot2^{*}_{\alpha}}
\int_{\mathbb{R}^{N}}
\int_{\mathbb{R}^{N}}
\frac{|u(x)|^{2^{*}_{\alpha}}|u(y)|^{2^{*}_{\alpha}}}{|x-y|^{\alpha}}
\mathrm{d}x
\mathrm{d}y
-
\sum_{i=1}^{k}
\frac{1}{2^{*}_{s,\theta_{i}}}
\int_{\mathbb{R}^{N}}
\frac{|u|^{2^{*}_{s,\theta_{i}}}}
{|x|^{\theta_{i}}}
\mathrm{d}x,\\
I_{3}(u)
=&
\frac{1}{2}
\|u\|_{D}^{2}
-
\frac{1}{2\cdot2^{*}_{\alpha}}
\int_{\mathbb{R}^{N}}
\int_{\mathbb{R}^{N}}
\frac{|u(x)|^{2^{*}_{\alpha}}|u(y)|^{2^{*}_{\alpha}}}{|x-y|^{\alpha}}
\mathrm{d}x
\mathrm{d}y.
\end{aligned}
\end{equation*}
The Nehari manifolds associated with
problem $(\mathcal{P}_{i})$
$(i=1,2,3)$,
which are  defined by
$$\mathcal{N}^{i}=\{u\in D^{s,2}(\mathbb{R}^{N})|\langle I^{'}_{i}(u),u\rangle=0,~u\not=0 \},$$
and
$$
c_{0}^{i}
=\inf_{u\in\mathcal{N}^{i}}
I_{i}(u),
~
c_{1}^{i}
=\inf_{u\in D^{s,2}(\mathbb{R}^{N})}\max_{t\geqslant0}
I_{i}(tu)
~\mathrm{and}~
c^{i}
=
\inf_{\Upsilon^{i}\in\Gamma^{i}}
\max_{t\in [0,1]}
I_{i}(\Upsilon^{i}(t)),$$
where
$
\Gamma^{i}=
\{
\Upsilon^{i}\in C([0,1],D^{s,2}(\mathbb{R}^{N}))
:
\Upsilon^{i}(0)=0,
I_{i}(\Upsilon^{i}(1))<0
\}
$.
\section{Some key Lemmas}
\noindent
We show the refinement of Hardy-Littlewood-Sobolev inequality.
\begin{lemma}\label{lemma5}
For any
$s\in(0,\frac{N}{2})$
and
$\alpha\in(0,N)$,
there exists
$C_{3}>0$
such that
for
$\iota$
and
$\vartheta$
satisfying
$\frac{2}{2^{*}_{s}}\leqslant\iota<1$,
$1\leqslant \vartheta<2^{*}_{s}=\frac{2N}{N-2s}$,
we have
\begin{align*}
\left(
\int_{\mathbb{R}^{N}}
\int_{\mathbb{R}^{N}}
\frac{|u(x)|^{2^{*}_{\alpha}}|u(y)|^{2^{*}_{\alpha}}}{|x-y|^{\alpha}}
\mathrm{d}x
\mathrm{d}y
\right)^{\frac{1}{2^{*}_{\alpha}}}
\leqslant
C_{3}
\|u\|_{D}^{2\iota}
\|u\|_{\mathcal{L}^{\vartheta,\frac{\vartheta(N-2s)}{2}}(\mathbb{R}^{N})}^{2(1-\iota)},
\end{align*}
for any
$u\in D^{s,2}(\mathbb{R}^{N})$.
\end{lemma}
\noindent
{\bf Proof.}
Let
$\frac{2}{2^{*}_{s}}\leqslant\iota<1$
and
$1\leqslant \vartheta<2^{*}_{s}=\frac{2N}{N-2s}$
.
By Hardy-Littlewood-Sobolev inequality
and
Lemma \ref{lemma4},
we obtain
\begin{equation*}
\begin{aligned}
\left(
\int_{\mathbb{R}^{N}}
\int_{\mathbb{R}^{N}}
\frac{|u(x)|^{2^{*}_{\alpha}}|u(y)|^{2^{*}_{\alpha}}}{|x-y|^{\alpha}}
\mathrm{d}x
\mathrm{d}y
\right)^{\frac{1}{2^{*}_{\alpha}}}
\leqslant&
C(N,\alpha)^{\frac{1}{2^{*}_{\alpha}}}
\|u\|_{L^{2^{*}_{s}}(\mathbb{R}^{N})}^{2}\\
\leqslant&
C(N,\alpha)^{\frac{1}{2^{*}_{\alpha}}}
C_{2}^{2}
\|u\|_{D}^{2\iota}
\|u\|_{\mathcal{L}^{\vartheta,\frac{\vartheta(N-2s)}{2}}(\mathbb{R}^{N})}^{2(1-\iota)}.
\end{aligned}
\end{equation*}
\qed

We show some  properties of fractional Coulomb--Sobolev space
$\mathcal{E}^{s,\alpha,2^{*}_{\alpha}}(\mathbb{R}^{N})$.
\begin{lemma}\label{lemma6}
Let
$( H_{1})$
hold.
If
$u\in \mathcal{E}^{s,\alpha_{j},2^{*}_{\alpha_{j}}}(\mathbb{R}^{N})$
$(j=1,\ldots,k)$,
then

\noindent
(i) $\|\cdot\|_{D}$
is an equivalent norm in
$\mathcal{E}^{s,\alpha_{j},2^{*}_{\alpha_{j}}}(\mathbb{R}^{N})$;

\noindent
(ii)
$
u\in
\bigcap_{i=1,i\not=j}^{k}
\mathcal{E}^{s,\alpha_{i},2^{*}_{\alpha_{i}}}(\mathbb{R}^{N})$;

\noindent
(iii)
$\|\cdot\|_{\mathcal{E},\alpha_{i}}$
are equivalent norms in
$\mathcal{E}^{s,\alpha_{j},2^{*}_{\alpha_{j}}}(\mathbb{R}^{N})$,
where
$i\not=j$
and
$i=1,\ldots,k$.
\end{lemma}
\noindent
{\bf Proof.}
\noindent
{\bf (1).}
Set
$j=1,\ldots,k$.
For any $u\in\mathcal{E}^{s,\alpha_{j},2^{*}_{\alpha{j}}}(\mathbb{R}^{N})$,
applying the definition of fractional Coulomb--Sobolev space,
we know
\begin{align}\label{14a}
\|u\|_{D}^{2}
\leqslant
\|u\|_{\mathcal{E},\alpha_{j}}^{2}
<\infty.
\end{align}
This implies that
$\mathcal{E}^{s,\alpha_{j},2^{*}_{\alpha_{j}}}(\mathbb{R}^{N})\subset D^{s,2}(\mathbb{R}^{N})$.

\noindent
According to  $\mathcal{E}^{s,\alpha_{j},2^{*}_{\alpha_{j}}}(\mathbb{R}^{N})\subset D^{s,2}(\mathbb{R}^{N})$
and \eqref{10},
we have
\begin{align}\label{14b}
\|u\|_{\mathcal{E},\alpha_{j}}^{2}
\leqslant
\left(
1
+
\frac{1}{S_{0,\alpha_{j}}}
\right)
\|u\|_{D}^{2}.
\end{align}
Combining
\eqref{14a}
and
\eqref{14b},
we obtain
\begin{align}\label{14c}
\|u\|_{D}^{2}
\leqslant
\|u\|_{\mathcal{E},\alpha_{j}}^{2}
\leqslant
\left(
1
+
\frac{1}{S_{0,\alpha_{j}}}
\right)
\|u\|_{D}^{2}.
\end{align}
These imply that
$\|\cdot\|_{D}$
is an equivalent norm in
$\mathcal{E}^{s,\alpha_{j},2^{*}_{\alpha_{j}}}(\mathbb{R}^{N})$.

\noindent
{\bf (2).}
For any $u\in\mathcal{E}^{s,\alpha_{j},2^{*}_{\alpha{j}}}(\mathbb{R}^{N})\subset D^{s,2}(\mathbb{R}^{N})$,
by using \eqref{14a}
and
\eqref{10},
we know
\begin{align}\label{14d}
S_{0,\alpha_{i}}
\left(
\int_{\mathbb{R}^{N}}
\int_{\mathbb{R}^{N}}
\frac{|u(x)|^{2^{*}_{\alpha_{i}}}|u(y)|^{2^{*}_{\alpha_{i}}}}{|x-y|^{\alpha_{i}}}
\mathrm{d}x
\mathrm{d}y
\right)^{\frac{1}{2^{*}_{\alpha_{i}}}}
\leqslant
\|u\|_{D}^{2}
\leqslant
\|u\|_{\mathcal{E},\alpha_{j}}^{2}
<\infty,
\end{align}
where
$i\not=j$
and
$i=1,\ldots,k$.
The inequality
\eqref{14d}
gives that
\begin{align*}
\|u\|_{\mathcal{E},\alpha_{i}}^{2}
=
\|u\|_{D}^{2}
+
\left(
\int_{\mathbb{R}^{N}}
\int_{\mathbb{R}^{N}}
\frac{|u(x)|^{2^{*}_{\alpha_{i}}}|u(y)|^{2^{*}_{\alpha_{i}}}}{|x-y|^{\alpha_{i}}}
\mathrm{d}x
\mathrm{d}y
\right)^{\frac{1}{2^{*}_{\alpha_{i}}}}
<\infty.
\end{align*}
This implies that $u\in\bigcap_{i=1,i\not=j}^{k}\mathcal{E}^{s,\alpha_{i},2^{*}_{\alpha{i}}}(\mathbb{R}^{N})$.

\noindent
{\bf (3).}
For any $u\in\mathcal{E}^{s,\alpha_{j},2^{*}_{\alpha{j}}}(\mathbb{R}^{N})$,
by using \eqref{14b},
we have
\begin{equation*}
\begin{aligned}
\|u\|_{\mathcal{E},\alpha_{j}}^{2}
\leqslant
\left(
\frac{S_{0,\alpha_{j}}+1}{S_{0,\alpha_{j}}}
\right)
\|u\|_{D}^{2}
\leqslant
\left(
\frac{S_{0,\alpha_{j}}+1}{S_{0,\alpha_{j}}}
\right)
\|u\|_{\mathcal{E},\alpha_{i}}^{2},
\end{aligned}
\end{equation*}
which imply that
\begin{equation}\label{15}
\begin{aligned}
\left(
\frac{S_{0,\alpha_{j}}}{S_{0,\alpha_{j}}+1}
\right)
\|u\|_{\mathcal{E},\alpha_{i}}^{2}
\leqslant
\|u\|_{\mathcal{E},\alpha_{j}}^{2}
\leqslant
\left(
\frac{S_{0,\alpha_{j}}+1}{S_{0,\alpha_{j}}}
\right)
\|u\|_{\mathcal{E},\alpha_{i}}^{2},
\end{aligned}
\end{equation}
where
$0<\frac{S_{0,\alpha_{j}}}{S_{0,\alpha_{j}}+1}
<1<
\frac{S_{0,\alpha_{j}}+1}{S_{0,\alpha_{j}}}<\infty
$.
\qed
\begin{lemma}\label{lemma7}
[Endpoint refined Hardy--Sobolev inequality]
Let
$s\in(0,\frac{N}{2})$,
$\theta\in(0,2s)$
and
$\alpha\in(0,N)$.
Then there exists a  constant
$C_{4}>0$
such that the inequality
\begin{equation*}
\begin{aligned}
\int_{\mathbb{R}^{N}}
\frac{|u|^{2^{*}_{s,\theta}}}
{|x|^{\theta}}
\mathrm{d}x
\leqslant
C_{4}
\|u\|_{D}^{\frac{2(N+\theta-\alpha)}{N+2s-\alpha}}
\left(
\int_{\mathbb{R}^{N}}
\int_{\mathbb{R}^{N}}
\frac{|u(x)|^{2^{*}_{\alpha}}|u(y)|^{2^{*}_{\alpha}}}{|x-y|^{\alpha}}
\mathrm{d}x
\mathrm{d}y
\right)
^{\frac{2s-\theta}{N+2s-\alpha}},
\end{aligned}
\end{equation*}
holds for all
$u\in \mathcal{E}^{s,\alpha,2^{*}_{\alpha}}(\mathbb{R}^{N})$.
\end{lemma}
\noindent
{\bf Proof.}
For any
$u\in \mathcal{E}^{s,\alpha,2^{*}_{\alpha}}(\mathbb{R}^{N})$.
By using H\"{o}lder inequality and fractional Hardy inequality,
we obtain
\begin{equation}\label{16}
\begin{aligned}
\int_{\mathbb{R}^{N}}
\frac{|u|^{2^{*}_{s,\theta}}}
{|x|^{\theta}}
\mathrm{d}x
=&
\int_{\mathbb{R}^{N}}
\frac{|u|^{\frac{\theta}{s}}}
{|x|^{\theta}}
\cdot
|u|^{\frac{2(N-\theta)}{N-2s}-\frac{\theta}{s}}
\mathrm{d}x\\
\leqslant&
\left(
\int_{\mathbb{R}^{N}}
\frac{|u|^{\frac{\theta}{s}\cdot\frac{2s}{\theta}}}
{|x|^{\theta\cdot\frac{2s}{\theta}}}
\mathrm{d}x
\right)
^{\frac{\theta}{2s}}
\left(
\int_{\mathbb{R}^{N}}
|u|^{\frac{(2s-\theta)N}{(N-2s)s}\cdot\frac{2s}{2s-\theta}}
\mathrm{d}x
\right)
^{1-\frac{\theta}{2s}}
\\
=&
\left(
\int_{\mathbb{R}^{N}}
\frac{|u|^{2}}
{|x|^{2s}}
\mathrm{d}x
\right)
^{\frac{\theta}{2s}}
\left(
\int_{\mathbb{R}^{N}}
|u|^{\frac{2N}{N-2s}}
\mathrm{d}x
\right)
^{\frac{2s-\theta}{2s}}
\\
\leqslant&
\left(
\frac{1}{\Lambda}
\right)
^{\frac{\theta}{2s}}
\|u\|_{D}
^{\frac{\theta}{s}}
\|u\|_{L^{2^{*}_{s}}(\mathbb{R}^{N})}^{\frac{N(2s-\theta)}{s(N-2s)}}.
\end{aligned}
\end{equation}
According to
Lemma \ref{lemma2}
and
(\ref{16}),
we know
\begin{equation*}
\begin{aligned}
&\int_{\mathbb{R}^{N}}
\frac{|u|^{2^{*}_{s,\theta}}}
{|x|^{\theta}}
\mathrm{d}x\\
\leqslant&
C_{4}
\|u\|_{D}
^{\frac{\theta}{s}}
\|u\|_{D}^{\frac{(N-\alpha)(2s-\theta)}{s(N+2s-\alpha)}}
\left(
\int_{\mathbb{R}^{N}}
\int_{\mathbb{R}^{N}}
\frac{|u(x)|^{2^{*}_{\alpha}}|u(y)|^{2^{*}_{\alpha}}}{|x-y|^{\alpha}}
\mathrm{d}x
\mathrm{d}y
\right)
^{\frac{2s-\theta}{N+2s-\alpha}}\\
=&
C_{4}
\|u\|_{D}^{\frac{2(N+\theta-\alpha)}{N+2s-\alpha}}
\left(
\int_{\mathbb{R}^{N}}
\int_{\mathbb{R}^{N}}
\frac{|u(x)|^{2^{*}_{\alpha}}|u(y)|^{2^{*}_{\alpha}}}{|x-y|^{\alpha}}
\mathrm{d}x
\mathrm{d}y
\right)
^{\frac{2s-\theta}{N+2s-\alpha}}.
\end{aligned}
\end{equation*}
\qed

We study the refinement of Hardy--Sobolev inequality.
In \cite{Yang2015,Yang2017},
the authors also obtained the Refinement of Hardy--Sobolev inequality.
However,
their parameter $\tilde{\vartheta}$
satisfying (see \cite[Theorem 1]{Yang2015})
$$1\leqslant \tilde{\vartheta}<2^{*}_{s,\theta}.$$
It is easy to see that
$$
2^{*}_{s,\theta}
=\frac{2(N-\theta)}{N-2s}
<
\frac{2N}{N-2s}
=
2^{*}_{s},$$
for
$s\in (0,\frac{N}{2})$
and
$\theta\in(0,2s)$.
{\bf It is natural to ask the case of
$\tilde{\vartheta}\in[2^{*}_{s,\theta},2^{*}_{s})$}.
Our method extends the parameter
$\tilde{\vartheta}$
from
$[1,2^{*}_{s,\theta})$
to
$[1,2^{*}_{s})$
.
\begin{lemma}\label{lemma8}
$\left.\right.$
[Refinement of Hardy--Sobolev inequality]
For
$s\in (0,\frac{N}{2})$
and
$\theta\in(0,2s)$,
there exists
$C_{5}>0$
such that
for
$\iota$
and
$\vartheta$
satisfying
$\frac{2}{2^{*}_{s}}\leqslant\iota<1$,
\textcolor{red}{$1\leqslant \vartheta<2^{*}_{s}$},
we have
\begin{align*}
\left(
\int_{\mathbb{R}^{N}}
\frac{|u|^{2^{*}_{s,\theta}}}{|x|^{\theta}}
\mathrm{d}x
\right)^{\frac{1}{2^{*}_{s,\theta}}}
\leqslant
C_{5}
\|u\|_{D}
^{\frac{\theta(N-2s)+\iota N(2s-\theta)}{2s(N-\theta)}}
\|u\|_{\mathcal{L}^{\vartheta,\frac{\vartheta(N-2s)}{2}}(\mathbb{R}^{N})}^{\frac{N(1-\iota)(2s-\theta)}{2s(N-\theta)}},
\end{align*}
for any
$u\in D^{s,2}(\mathbb{R}^{N})$.
\end{lemma}
\begin{proof}
Combining (\ref{16}) and Lemma \ref{lemma4},
we have
\begin{equation*}
\begin{aligned}
\left(
\int_{\mathbb{R}^{N}}
\frac{|u|^{2^{*}_{s,\theta}}}{|x|^{\theta}}
\mathrm{d}x
\right)^{\frac{1}{2^{*}_{s,\theta}}}
\leqslant&
\left(
\frac{1}{\Lambda}
\right)
^{\frac{\theta(N-2s)}{4s(N-\theta)}}
\|u\|_{D}
^{\frac{\theta(N-2s)}{2s(N-\theta)}}
\|u\|_{L^{2^{*}_{s}}(\mathbb{R}^{N})}^{\frac{N(2s-\theta)}{2s(N-\theta)}}\\
\leqslant&
\left(
\frac{1}{\Lambda}
\right)
^{\frac{\theta(N-2s)}{4s(N-\theta)}}
\|u\|_{D}
^{\frac{\theta(N-2s)}{2s(N-\theta)}}
\left(
C_{2}
\|u\|_{D}^{\iota}
\|u\|_{\mathcal{L}^{\vartheta,\frac{\vartheta(N-2s)}{2}}}^{1-\iota}
\right)^{\frac{N(2s-\theta)}{2s(N-\theta)}}\\
=&
C_{5}
\|u\|_{D}
^{\frac{\theta(N-2s)+\iota N(2s-\theta)}{2s(N-\theta)}}
\|u\|_{\mathcal{L}^{\vartheta,\frac{\vartheta(N-2s)}{2}}(\mathbb{R}^{N})}
^{\frac{N(1-\iota)(2s-\theta)}{2s(N-\theta)}}.
\end{aligned}
\end{equation*}
\end{proof}

\begin{lemma}\label{lemma9}
Let
$s\in(0,\frac{N}{2})$
and
$0<\theta<\tilde{\theta}<2s$.
Then the inequality
\begin{equation*}
\begin{aligned}
\int_{\mathbb{R}^{N}}
\frac{|u|^{2^{*}_{s,\theta}}}
{|x|^{\theta}}
\mathrm{d}x
\leqslant
\left(
\int_{\mathbb{R}^{N}}
\frac{|u|^{2^{*}_{s,\tilde{\theta}}}}
{|x|^{\tilde{\theta}}}
\mathrm{d}x
\right)
^{\frac{\theta}{\tilde{\theta}}}
\left(
\int_{\mathbb{R}^{N}}
|u|^{2^{*}_{s}}
\mathrm{d}x
\right)
^{\frac{\tilde{\theta}-\theta}{\tilde{\theta}}},
\end{aligned}
\end{equation*}
holds for all
$u\in D^{s,2}(\mathbb{R}^{N})$.
\end{lemma}
\noindent
{\bf Proof.}
For any
$u\in D^{s,2}(\mathbb{R}^{N})$.
By using H\"{o}lder inequality and $0<\theta<\tilde{\theta}<2s$,
we obtain
\begin{equation*}
\begin{aligned}
\int_{\mathbb{R}^{N}}
\frac{|u|^{2^{*}_{s,\theta}}}
{|x|^{\theta}}
\mathrm{d}x
=&
\int_{\mathbb{R}^{N}}
\frac{|u|^{\frac{\theta}{\tilde{\theta}}\cdot\frac{2(N-\tilde{\theta})}{N-2s}}}
{|x|^{\theta}}
\cdot
|u|^{\frac{2N}{N-2s}\cdot\frac{\tilde{\theta}-\theta}{\tilde{\theta}}}
\mathrm{d}x\\
\leqslant&
\left(
\int_{\mathbb{R}^{N}}
\frac{|u|^{\frac{\theta}{\tilde{\theta}}\cdot\frac{2(N-\tilde{\theta})}{N-2s}\cdot\frac{\tilde{\theta}}{\theta}}}
{|x|^{\theta\cdot\frac{\tilde{\theta}}{\theta}}}
\mathrm{d}x
\right)
^{\frac{\theta}{\tilde{\theta}}}
\left(
\int_{\mathbb{R}^{N}}
|u|^{\frac{2N}{N-2s}\cdot\frac{\tilde{\theta}-\theta}{\tilde{\theta}}\cdot\frac{\tilde{\theta}}{\tilde{\theta}-\theta}}
\mathrm{d}x
\right)
^{1-\frac{\theta}{\tilde{\theta}}}
\\
=&
\left(
\int_{\mathbb{R}^{N}}
\frac{|u|^{\frac{2(N-\tilde{\theta})}{N-2s}}}
{|x|^{\tilde{\theta}}}
\mathrm{d}x
\right)
^{\frac{\theta}{\tilde{\theta}}}
\left(
\int_{\mathbb{R}^{N}}
|u|^{\frac{2N}{N-2s}}
\mathrm{d}x
\right)
^{\frac{\tilde{\theta}-\theta}{\tilde{\theta}}}.
\end{aligned}
\end{equation*}
\qed

\begin{lemma}\label{lemma10}
Let
$s\in(0,\frac{N}{2})$,
$0<\bar{\theta}<\theta<2s$
and
$2\theta-\bar{\theta}<2s$.
Then the inequality
\begin{equation*}
\begin{aligned}
\int_{\mathbb{R}^{N}}
\frac{|u|^{2^{*}_{s,\theta}}}
{|x|^{\theta}}
\mathrm{d}x
\leqslant
\left(
\int_{\mathbb{R}^{N}}
\frac{|u|^{2^{*}_{s,\bar{\theta}}}}
{|x|^{\bar{\theta}}}
\mathrm{d}x
\right)
^{\frac{1}{2}}
\left(
\int_{\mathbb{R}^{N}}
\frac{|u|^{2^{*}_{s,2\theta-\bar{\theta}}}}
{|x|^{2\theta-\bar{\theta}}}
\mathrm{d}x
\right)
^{\frac{1}{2}},
\end{aligned}
\end{equation*}
holds for all
$u\in D^{s,2}(\mathbb{R}^{N})$.
\end{lemma}
\noindent
{\bf Proof.}
For any
$u\in D^{s,2}(\mathbb{R}^{N})$.
By using H\"{o}lder inequality and $0<\bar{\theta}<\theta<2s$,
we obtain
\begin{equation*}
\begin{aligned}
\int_{\mathbb{R}^{N}}
\frac{|u|^{2^{*}_{s,\theta}}}
{|x|^{\theta}}
\mathrm{d}x
=&
\int_{\mathbb{R}^{N}}
\frac{|u|^{\frac{N-\bar{\theta}}{N-2s}}}
{|x|^{\frac{\bar{\theta}}{2}}}
\cdot
\frac{|u|^{\frac{N-(2\theta-\bar{\theta})}{N-2s}}}
{|x|^{\theta-\frac{\bar{\theta}}{2}}}
\mathrm{d}x\\
\leqslant&
\left(
\int_{\mathbb{R}^{N}}
\frac{|u|^{\frac{2(N-\bar{\theta})}{N-2s}}}
{|x|^{\bar{\theta}}}
\mathrm{d}x
\right)
^{\frac{1}{2}}
\left(
\int_{\mathbb{R}^{N}}
\frac{|u|^{\frac{2[N-(2\theta-\bar{\theta})]}{N-2s}}}
{|x|^{2\theta-\bar{\theta}}}
\mathrm{d}x
\right)
^{\frac{1}{2}}.
\end{aligned}
\end{equation*}
Since
$0<2\theta-\bar{\theta}<2s$,
we get
\begin{equation*}
\begin{aligned}
\int_{\mathbb{R}^{N}}
\frac{|u|^{2^{*}_{s,\theta}}}
{|x|^{\theta}}
\mathrm{d}x
\leqslant
\left(
\int_{\mathbb{R}^{N}}
\frac{|u|^{2^{*}_{s,\bar{\theta}}}}
{|x|^{\bar{\theta}}}
\mathrm{d}x
\right)
^{\frac{1}{2}}
\left(
\int_{\mathbb{R}^{N}}
\frac{|u|^{2^{*}_{s,2\theta-\bar{\theta}}}}
{|x|^{2\theta-\bar{\theta}}}
\mathrm{d}x
\right)
^{\frac{1}{2}}.
\end{aligned}
\end{equation*}
\qed
\begin{lemma}\label{lemma22}
Let
$N\geqslant3$,
$\alpha\in(0,   N)$,
$s\in(0,1)$
and
$\zeta\in[0,\Lambda)$
hold.
Then we have
\begin{equation*}
\begin{aligned}
S_{0,\alpha}\geqslant
S_{\zeta,\alpha}
\geqslant
\left(1-\frac{\zeta}{\Lambda}\right)
S_{0,\alpha},
\end{aligned}
\end{equation*}
and
\begin{equation*}
\begin{aligned}
H_{0,\theta}
\geqslant
H_{\zeta,\theta}
\geqslant
\left(1-\frac{\zeta}{\Lambda}\right)
H_{0,\theta}.
\end{aligned}
\end{equation*}
\end{lemma}
\noindent
{\bf Proof.}
For
$\zeta\in[0,\Lambda)$
and
$u\in D^{s,2}(\mathbb{R}^{N})$,
$u\not\equiv0$.
We set
\begin{equation*}
\begin{aligned}
F_{\zeta}(u):=
\frac{\|u\|_{\zeta}^{2}}
{\left(
\int_{\mathbb{R}^{N}}
\int_{\mathbb{R}^{N}}
\frac{|u(x)|^{2^{*}_{\alpha}}|u(y)|^{2^{*}_{\alpha}}}{|x-y|^{\alpha}}
\mathrm{d}x
\mathrm{d}y
\right)^{\frac{1}{2^{*}_{\alpha}}}},
\end{aligned}
\end{equation*}
clearly,
for a fixed $u$,
$F_{\zeta}(u)$
is decreasing with respect to
$\zeta$.

Moreover, for any fixed
$\zeta\in[0,\Lambda)$,
we denote by
$u_{\zeta}\in D^{s,2}(\mathbb{R}^{N})$
a function such that \eqref{10} is achieved,
that is
$S_{\zeta,\alpha}=F_{\zeta}(u_{\zeta})$.

Let
$0<\zeta_{1}<\zeta_{2}<\Lambda$.
Then
\begin{equation*}
\begin{aligned}
S_{\zeta_{1},\alpha}=
F_{\zeta_{1}}(u_{\zeta_{1}})
=&
\frac{\|u_{\zeta_{1}}\|_{D}^{2}
-
\zeta_{1}
\int_{\mathbb{R}^{N}}
\frac{ |u_{\zeta_{1}}|^{2}}{|x|^{2s}}
\mathrm{d}x}
{\left(
\int_{\mathbb{R}^{N}}
\int_{\mathbb{R}^{N}}
\frac{|u_{\zeta_{1}}(x)|^{2^{*}_{\alpha}}|u_{\zeta_{1}}(y)|^{2^{*}_{\alpha}}}{|x-y|^{\alpha}}
\mathrm{d}x
\mathrm{d}y
\right)^{\frac{1}{2^{*}_{\alpha}}}}\\
>&
\frac{\|u_{\zeta_{1}}\|_{D}^{2}
-
\zeta_{2}
\int_{\mathbb{R}^{N}}
\frac{ |u_{\zeta_{1}}|^{2}}{|x|^{2s}}
\mathrm{d}x}
{\left(
\int_{\mathbb{R}^{N}}
\int_{\mathbb{R}^{N}}
\frac{|u_{\zeta_{1}}(x)|^{2^{*}_{\alpha}}|u_{\zeta_{1}}(y)|^{2^{*}_{\alpha}}}{|x-y|^{\alpha}}
\mathrm{d}x
\mathrm{d}y
\right)^{\frac{1}{2^{*}_{\alpha}}}}
\geqslant
S_{\zeta_{2},\alpha}.
\end{aligned}
\end{equation*}
Let
$0<\zeta<\Lambda$.
Since  the best constant in the Hardy inequality is not achieved,
we get
\begin{equation*}
\begin{aligned}
S_{\zeta,\alpha}=
F_{\zeta}(u_{\zeta})
>&
\left(1-\frac{\zeta}{\Lambda}\right)
\frac{\|u_{\zeta}\|_{D}^{2}}
{\left(
\int_{\mathbb{R}^{N}}
\int_{\mathbb{R}^{N}}
\frac{|u_{\zeta}(x)|^{2^{*}_{\alpha}}|u_{\zeta}(y)|^{2^{*}_{\alpha}}}{|x-y|^{\alpha}}
\mathrm{d}x
\mathrm{d}y
\right)^{\frac{1}{2^{*}_{\alpha}}}}\\
\geqslant&
\left(1-\frac{\zeta}{\Lambda}\right)
S_{0,\alpha}.
\end{aligned}
\end{equation*}
\qed
\section{The proof of theorem \ref{theorem2}}
In this section, we show the existence of nonnegative solution of problems $(\mathcal{P}_{2})$.

In Lemma \ref{lemma11}--Lemma \ref{lemma14},
we will prove some properties of the Nehari manifolds associated with
problems $(\mathcal{P}_{2})$ and $(\mathcal{P}_{4})$.
\begin{lemma}\label{lemma11}
Assume that the assumptions of Theorem
\ref{theorem2}
hold.
Then
$$c_{0}^{2}=\inf_{u\in\mathcal{N}^{2}}I_{2}(u)>0.$$
\end{lemma}
\noindent
{\bf Proof.}
\noindent
{\bf Step 1.}
We claim that
any limit point of a sequence in
$\mathcal{N}^{2}$
is different from zero.
According to
$\langle I^{'}_{2}(u),u\rangle=0$,
(\ref{10})
and
(\ref{11}),
for any $u\in\mathcal{N}^{2}$,
we obtain
\begin{equation*}
\begin{aligned}
0
=
\langle I^{'}_{2}(u),u\rangle
\geqslant&
\|u\|^{2}_{\zeta}
-
\frac{1}{S_{\zeta,\alpha} ^{2^{*}_{\alpha}}}
\|u\|_{\zeta}^{2\cdot 2^{*}_{\alpha} }
-
\sum_{i=1}^{k}
\frac{1}{H_{\zeta,\theta_{i}}^{\frac{2^{*}_{s,\theta_{i}} }{2}}}
\|u\|_{\zeta}^{2^{*}_{s,\theta_{i}} }.
\end{aligned}
\end{equation*}
From above expression,
we have
\begin{equation}\label{17}
\begin{aligned}
\|u\|^{2}_{\zeta}
\leqslant
\frac{1}{S_{\zeta,\alpha} ^{2^{*}_{\alpha}}}
\|u\|_{\zeta}^{2\cdot 2^{*}_{\alpha} }
+
\sum_{i=1}^{k}
\frac{1}{H_{\zeta,\theta_{i}}^{\frac{2^{*}_{s,\theta_{i}} }{2}}}
\|u\|_{\zeta}^{2^{*}_{s,\theta_{i}} }.
\end{aligned}
\end{equation}
Set
$$\kappa
:=
\frac{1}{S_{\zeta,\alpha} ^{2^{*}_{\alpha}}}
+
\sum_{i=1}^{k}
\frac{1}{H_{\zeta,\theta_{i}}^{\frac{2^{*}_{s,\theta_{i}} }{2}}}
.$$
Applying
(\ref{10})
and
(\ref{11}),
we get
$$
0<\kappa<\infty.$$
From
($H_{1}$),
we know
$$2^{*}_{s,\theta_{k}}<\cdots<2^{*}_{s,\theta_{1}}<2\cdot2^{*}_{\alpha}.$$
Now the proof of Step 1 is divided into two cases:
(i)
$\|u\|_{\zeta}\geqslant1$;
(ii)
$\|u\|_{\zeta}<1$.

\noindent
{\bf Case (i)}$\|u\|_{\zeta}\geqslant1$.
From
(\ref{15}),
we have
\begin{equation*}
\begin{aligned}
\|u\|_{\zeta}^{2}
\leqslant&
\kappa
\|u\|_{\zeta}^{2\cdot2^{*}_{\alpha}},
\end{aligned}
\end{equation*}
which implies that
\begin{equation}\label{18}
\begin{aligned}
\|u\|_{\zeta}
\geqslant
\kappa
^{\frac{1}{2-2\cdot2^{*}_{\alpha}}}.
\end{aligned}
\end{equation}
\noindent
{\bf Case (ii)}$\|u\|_{\zeta}<1$.
From
(\ref{17}),
we know
\begin{equation}\label{19}
\begin{aligned}
\|u\|_{\zeta}
\geqslant
\kappa
^{\frac{1}{2-2^{*}_{s,\theta_{k}} }}.
\end{aligned}
\end{equation}
Combining the Cases (i) and (ii),
according to
(\ref{18})
and
(\ref{19}),
we deduce that
\begin{equation}\label{20}
\begin{aligned}
\|u\|_{\zeta}
\geqslant
\begin{cases}
\kappa
^{\frac{1}{2-2\cdot2^{*}_{\alpha} }},
&\kappa<1,\\
\kappa
^{\frac{1}{2-2^{*}_{s,\theta_{k}} }},
&\kappa\geqslant1.
\end{cases}
\end{aligned}
\end{equation}
Hence,
we know that
any limit point of a sequence in $\mathcal{N}^{2}$ is different from zero.

\noindent
{\bf Step 2.}
Now,
we claim that
$I_{2}$
is bounded from below on
$\mathcal{N}^{2}$.
For any
$u\in \mathcal{N}^{2}$,
by using
(\ref{20}),
we get
\begin{equation*}
\begin{aligned}
I_{2}(u)
\geqslant&
\left(
\frac{1}{2}
-
\frac{1}{2^{*}_{s,\theta_{k}}}
\right)
\|u\|^{2}_{\zeta}
\geqslant
\begin{cases}
\left(
\frac{1}{2}
-
\frac{1}{2^{*}_{s,\theta_{k}}}
\right)
\kappa
^{\frac{2}{2-2\cdot2^{*}_{\alpha} }},
&\kappa\leqslant1,\\
\left(
\frac{1}{2}
-
\frac{1}{2^{*}_{s,\theta_{k}}}
\right)
\kappa
^{\frac{2}{2-2^{*}_{s,\theta_{k}}}},
&\kappa>1.
\end{cases}
\end{aligned}
\end{equation*}
Therefore,
$I_{2}$
is bounded from below on
$\mathcal{N}^{2}$,
and
$c_{0}^{2}>0$.
\qed
\begin{lemma}\label{lemma12}
Assume that the assumptions of Theorem
\ref{theorem2}
hold.
Then

\noindent
(i)
for each
$u\in D^{s,2}(\mathbb{R}^{N})\setminus\{0\}$,
there exists a unique
$t_{u}>0$
such that
$t_{u}u\in \mathcal{N}^{2}$;

\noindent
(ii)
$c_{0}^{2}=c_{1}^{2}=c^{2}>0$.
\end{lemma}
\noindent
{\bf Proof.}
The proof is standard, so we sketch it. Further details can be derived as in the proofs of Theorem 4.1 and 4.2 in
\cite{Willem1996},
we omit it.
\qed

Similar to Lemma \ref{lemma11} and Lemma \ref{lemma12},
we also have the following result.
\begin{lemma}\label{lemma14}
Assume that the assumptions of Theorem
\ref{theorem2}
hold.
For any
$u\in\mathcal{N}^{3}$,
we have
$$\|u\|_{D}^{2}\geqslant S_{0,\alpha}^{\frac{2N-\alpha}{N+2s-\alpha}},$$
and
$$c_{0}^{3}=\inf_{u\in\mathcal{N}^{3}}I_{3}(u)>0.$$
\end{lemma}
We show that the functional $I_{2}$ satisfies the Mountain--Pass geometry, and
estimate the Mountain--Pass levels.
\begin{lemma}\label{lemma15}
Assume that the assumptions of Theorem
\ref{theorem2}
hold.
Then
there exists a
$(PS)_{c^{2}}$
sequence of
$I_{2}$
at level
$c^{2}$,
where
$$0<c^{2}<c^{2,*}=
\min
\left\{
\frac{N+2s-\alpha}{2(2N-\alpha)}
S_{\zeta,\alpha}^{\frac{2N-\alpha}{N+2s-\alpha}}
,
\frac{2s-\theta_{1}}{2(N-\theta_{1})}
H_{\zeta,\theta_{1}} ^{\frac{N-\theta_{1}}{2s-\theta_{1}}}
,
\ldots
,
\frac{2s-\theta_{k}}{2(N-\theta_{k})}
H_{\zeta,\theta_{k}} ^{\frac{N-\theta_{k}}{2s-\theta_{k}}}
\right\}.$$
\end{lemma}
\noindent
{\bf Proof.}
The proof is standard, so we sketch it. Further details can be derived as in the proofs of Theorem 2 in \cite{Pucci2009}, we omit it.
\qed

The following result implies the non--vanishing of
$(PS)_{c^{2}}$
sequence.
\begin{lemma}\label{lemma16}
Assume that the assumptions of Theorem
\ref{theorem2}
hold.
Let
$\{u_{n}\}$
be a
$(PS)_{c^{2}}$
sequence of
$I_{2}$
with
$c^{2}\in(0,c^{2,*})$,
then
\begin{equation*}
\begin{aligned}
\lim_{n\rightarrow\infty}
\int_{\mathbb{R}^{N}}
\int_{\mathbb{R}^{N}}
\frac{|u_{n}(x)|^{2^{*}_{\alpha}}|u_{n}(y)|^{2^{*}_{\alpha}}}{|x-y|^{\alpha}}
\mathrm{d}x
\mathrm{d}y>0,
\end{aligned}
\end{equation*}
and
\begin{equation*}
\begin{aligned}
\lim_{n\rightarrow\infty}
\int_{\mathbb{R}^{N}}
\frac{|u_{n}|^{2^{*}_{s,\theta_{i}}}}
{|x|^{\theta_{i}}}
\mathrm{d}x
>0,(i=1,\ldots,k).
\end{aligned}
\end{equation*}
\end{lemma}
\noindent
{\bf Proof.}
The proof of this Lemma is divided into four cases:

\noindent
(1)
$
\lim_{n\rightarrow\infty}
\int_{\mathbb{R}^{N}}
\frac{|u_{n}|^{2^{*}_{s,\theta_{1}}}}
{|x|^{\theta_{1}}}
\mathrm{d}x
>0
$;

\noindent
(2)
$
\lim_{n\rightarrow\infty}
\int_{\mathbb{R}^{N}}
\frac{|u_{n}|^{2^{*}_{s,\theta_{k}}}}
{|x|^{\theta_{k}}}
\mathrm{d}x
>0
$;

\noindent
(3)
$
\lim_{n\rightarrow\infty}
\int_{\mathbb{R}^{N}}
\frac{|u_{n}|^{2^{*}_{s,\theta_{j}}}}
{|x|^{\theta_{j}}}
\mathrm{d}x
>0
$,
$(j=2,\ldots,k-1)$;

\noindent
(4)
$
\lim_{n\rightarrow\infty}
\int_{\mathbb{R}^{N}}
\int_{\mathbb{R}^{N}}
\frac{|u_{n}(x)|^{2^{*}_{\alpha}}|u_{n}(y)|^{2^{*}_{\alpha}}}{|x-y|^{\alpha}}
\mathrm{d}x
\mathrm{d}y>0
$.

\noindent
{\bf Case 1.}
It is easy to see that
$\{u_{n}\}$
is uniformly bounded in
$D^{s,2}(\mathbb{R}^{N})$.
Suppose  on the contrary that
\begin{equation}\label{21}
\begin{aligned}
\lim_{n\rightarrow\infty}
\int_{\mathbb{R}^{N}}
\frac{|u_{n}|^{2^{*}_{s,\theta_{1}}}}
{|x|^{\theta_{1}}}
\mathrm{d}x
=0.
\end{aligned}
\end{equation}
From
($H_{2}$),
we know
\begin{equation}\label{22}
\begin{aligned}
0<2\theta_{2}-\theta_{1}<\cdots<2\theta_{k}-\theta_{1}<2s.
\end{aligned}
\end{equation}
Since
$\{u_{n}\}$
is uniformly bounded in
$D^{s,2}(\mathbb{R}^{N})$,
there exists a constant
$0<C<\infty$
such that
$\|u_{n}\|_{D}\leqslant C$.
Applying
(\ref{22})
and
(\ref{11}),
we obtain
\begin{equation}\label{23}
\begin{aligned}
\lim_{n\rightarrow\infty}
\int_{\mathbb{R}^{N}}
\frac{|u_{n}|^{2^{*}_{s,2\theta_{i}-\theta_{1}}}}
{|x|^{2\theta_{i}-\theta_{1}}}
\mathrm{d}x
\leqslant C,
(i=2,\ldots,k).
\end{aligned}
\end{equation}
According to
Lemma
\ref{lemma10},
(\ref{21})
and
(\ref{23}),
we obtain
\begin{equation}\label{24}
\begin{aligned}
&\lim_{n\rightarrow\infty}
\int_{\mathbb{R}^{N}}
\frac{|u_{n}|^{2^{*}_{s,\theta_{i}}}}
{|x|^{\theta_{i}}}
\mathrm{d}x\\
\leqslant&
\left(
\lim_{n\rightarrow\infty}
\int_{\mathbb{R}^{N}}
\frac{|u_{n}|^{2^{*}_{s,\theta_{1}}}}
{|x|^{\theta_{1}}}
\mathrm{d}x
\right)
^{\frac{1}{2}}
\left(
\lim_{n\rightarrow\infty}
\int_{\mathbb{R}^{N}}
\frac{|u_{n}|^{2^{*}_{s,2\theta_{i}-\theta_{1}}}}
{|x|^{2\theta_{i}-\theta_{1}}}
\mathrm{d}x
\right)
^{\frac{1}{2}}\\
=&0
~(i=2,\ldots,k).
\end{aligned}
\end{equation}
By using
(\ref{21}),
(\ref{24})
and the definition of
$(PS)_{c^{2}}$
sequence,
we obtain
$$
c^{2}
+
o(1)
=
\frac{1}{2}
\|u_{n}\|_{\zeta}^{2}
-
\frac{1}{2\cdot2^{*}_{\alpha}}
\int_{\mathbb{R}^{N}}
\int_{\mathbb{R}^{N}}
\frac{|u_{n}(x)|^{2^{*}_{\alpha}}|u_{n}(y)|^{2^{*}_{\alpha}}}{|x-y|^{\alpha}}
\mathrm{d}x
\mathrm{d}y,$$
and
$$o(1)
=
\|u_{n}\|_{\zeta}^{2}
-
\int_{\mathbb{R}^{N}}
\int_{\mathbb{R}^{N}}
\frac{|u_{n}(x)|^{2^{*}_{\alpha}}|u_{n}(y)|^{2^{*}_{\alpha}}}{|x-y|^{\alpha}}
\mathrm{d}x
\mathrm{d}y.
$$
These yield
$$
c^{2}
+
o(1)
=
\frac{N+2s-\alpha}{2(2N-\alpha)}
\|u_{n}\|_{\zeta}^{2}.$$
Moreover,
\begin{equation*}
\begin{aligned}
S_{\zeta,\alpha}
\left(
\int_{\mathbb{R}^{N}}
\int_{\mathbb{R}^{N}}
\frac{|u_{n}(x)|^{2^{*}_{\alpha}}|u_{n}(y)|^{2^{*}_{\alpha}}}{|x-y|^{\alpha}}
\mathrm{d}x
\mathrm{d}y
\right)^{\frac{1}{2^{*}_{\alpha}}}
\leqslant
\|u_{n}\|^{2}_{\zeta},
\end{aligned}
\end{equation*}
which implies that
\begin{equation*}
\begin{aligned}
S_{\zeta,\alpha}^{\frac{2N-\alpha}{N+2s-\alpha}}
\leqslant
\|u_{n}\|^{2}_{\zeta}.
\end{aligned}
\end{equation*}
Therefore,
we obtain
\begin{equation*}
\begin{aligned}
\frac{N+2s-\alpha}{2(2N-\alpha)}
S_{\zeta,\alpha}^{\frac{2N-\alpha}{N+2s-\alpha}}
\leqslant
c^{2}.
\end{aligned}
\end{equation*}
This is a contradiction.

\noindent
{\bf Case 2.}
Suppose  on the contrary that
\begin{equation}\label{25}
\begin{aligned}
\lim_{n\rightarrow\infty}
\int_{\mathbb{R}^{N}}
\frac{|u_{n}|^{2^{*}_{s,\theta_{k}}}}
{|x|^{\theta_{k}}}
\mathrm{d}x
=0.
\end{aligned}
\end{equation}
By using
(\ref{10})
and
$\|u_{n}\|_{D}\leqslant C$,
we have
\begin{equation}\label{26}
\begin{aligned}
\lim_{n\rightarrow\infty}
\int_{\mathbb{R}^{N}}
|u_{n}|^{2^{*}_{s}}
\mathrm{d}x
\leqslant C.
\end{aligned}
\end{equation}
Applying
($H_{2}$),
Lemma
\ref{lemma9},
(\ref{25})
and
(\ref{26}),
we obtain
\begin{equation}\label{27}
\begin{aligned}
&
\lim_{n\rightarrow\infty}
\int_{\mathbb{R}^{N}}
\frac{|u_{n}|^{2^{*}_{s,\theta_{i}}}}
{|x|^{\theta_{i}}}
\mathrm{d}x\\
\leqslant&
\left(
\lim_{n\rightarrow\infty}
\int_{\mathbb{R}^{N}}
\frac{|u_{n}|^{2^{*}_{s,\theta_{k}}}}
{|x|^{\theta_{k}}}
\mathrm{d}x
\right)
^{\frac{\theta_{i}}{\theta_{k}}}
\left(
\lim_{n\rightarrow\infty}
\int_{\mathbb{R}^{N}}
|u_{n}|^{2^{*}_{s}}
\mathrm{d}x
\right)
^{\frac{\theta_{k}-\theta_{i}}{\theta_{k}}}\\
=&0
~(i=1,\ldots,k-1).
\end{aligned}
\end{equation}
By using
(\ref{25}),
(\ref{27})
and the definition of
$(PS)_{c^{2}}$
sequence,
similar to Case 1,
we get
$
\frac{N+2s-\alpha}{2(2N-\alpha)}
S_{\alpha}^{\frac{2N-\alpha}{N+2s-\alpha}}
\leqslant
c^{2}.
$
This is a contradiction.

\noindent
{\bf Case 3.}
Set
$j\in [2,k-1]$.
Suppose on the contrary that
\begin{equation}\label{}
\begin{aligned}
\lim_{n\rightarrow\infty}
\int_{\mathbb{R}^{N}}
\frac{|u_{n}|^{2^{*}_{s,\theta_{j}}}}
{|x|^{\theta_{j}}}
\mathrm{d}x
=0.
\end{aligned}
\end{equation}
From
($H_{2}$),
we know
\begin{equation*}
\begin{aligned}
0<2\theta_{j+1}-\theta_{j}<\cdots<2\theta_{k}-\theta_{j}<2\theta_{k}-\theta_{1}<2s.
\end{aligned}
\end{equation*}
Similar to
\eqref{24},
we obtain
\begin{equation*}
\begin{aligned}
\lim_{n\rightarrow\infty}
\int_{\mathbb{R}^{N}}
\frac{|u_{n}|^{2^{*}_{s,\theta_{i}}}}
{|x|^{\theta_{i}}}
\mathrm{d}x
=0
~(i=j+1,\ldots,k).
\end{aligned}
\end{equation*}
Similar to \eqref{27},
we have
\begin{equation*}
\begin{aligned}
\lim_{n\rightarrow\infty}
\int_{\mathbb{R}^{N}}
\frac{|u_{n}|^{2^{*}_{s,\theta_{i}}}}
{|x|^{\theta_{i}}}
\mathrm{d}x
=&0
~(i=1,\ldots,j-1).
\end{aligned}
\end{equation*}
Similar to Case 1,
we get
$
\frac{N+2s-\alpha}{2(2N-\alpha)}
S_{\alpha}^{\frac{2N-\alpha}{N+2s-\alpha}}
\leqslant
c^{2}.
$
This is a contradiction.

\noindent
{\bf Case 4.}
Suppose  on the contrary that
\begin{equation}\label{28}
\begin{aligned}
\lim_{n\rightarrow\infty}
\int_{\mathbb{R}^{N}}
\int_{\mathbb{R}^{N}}
\frac{|u_{n}(x)|^{2^{*}_{\alpha}}|u_{n}(y)|^{2^{*}_{\alpha}}}{|x-y|^{\alpha}}
\mathrm{d}x
\mathrm{d}y
=0.
\end{aligned}
\end{equation}
By using
$\|u_{n}\|_{D}\leqslant C$,
(\ref{10})
and the definition of
fractional Coulomb--Sobolev space,
we obtain
$u_{n}\in \mathcal{E}^{s,\alpha,2^{*}_{\alpha}}(\mathbb{R}^{N})$.
Applying Lemma \ref{lemma7} and (\ref{28}),
we have
\begin{equation}\label{29}
\begin{aligned}
&
\lim_{n\rightarrow\infty}
\int_{\mathbb{R}^{N}}
\frac{|u_{n}|^{2^{*}_{s,\theta_{i}}}}
{|x|^{\theta_{i}}}
\mathrm{d}x\\
\leqslant&
C
\left(
\lim_{n\rightarrow\infty}
\int_{\mathbb{R}^{N}}
\int_{\mathbb{R}^{N}}
\frac{|u_{n}(x)|^{2^{*}_{\alpha}}|u_{n}(y)|^{2^{*}_{\alpha}}}{|x-y|^{\alpha}}
\mathrm{d}x
\mathrm{d}y
\right)
^{\frac{2s-\theta_{i}}{N+2s-\alpha}}
=0
~(i=1,\ldots,k).
\end{aligned}
\end{equation}
Applying
(\ref{28})
and
(\ref{29}),
we get
$$
c^{2}
+
o(1)
=
\frac{1}{2}
\|u_{n}\|_{\zeta}^{2},
$$
and
$$o(1)
=
\|u_{n}\|_{\zeta}^{2},
$$
which imply that
$c^{2}=0$.
This contradicts with $c^{2}>0$.
\qed

In next lemma,
we show that
$c^{3}_{0}>c_{0}^{2}$.
This result plays a key role in the proof of Theorem \ref{theorem2}.
\begin{lemma}\label{lemma17}
Assume that the assumptions of Theorem
\ref{theorem2}
hold.
Then $c^{3}_{0}>c_{0}^{2}$.
\end{lemma}

\begin{proof}
Consider the family of functions
$\{U_{\sigma}\}$
defined as
\begin{equation*}
\begin{aligned}
U_{\sigma}(x)
=
\sigma^{-\frac{N-2s}{2}}
\left(
\frac{u^{*}(\frac{x}{\sigma})}{\|u^{*}\|_{2^{*}_{s}}}
\right),
\end{aligned}
\end{equation*}
where
$$u^{*}(x)=\frac{\varpi}
{(1+|x|^{2})^{\frac{N -2s}{2}}},
~\varpi\in \mathbb{R}\setminus\{0\}.$$
Then,
for each
$\sigma>0$,
$U_{\sigma}$
(see \cite{Mukherjee2017Fractional})
satisfies
\begin{equation*}
\begin{aligned}
(-\Delta)^{s}
U_{\sigma}
=
|U_{\sigma}|^{2^{*}_{s}-2}U_{\sigma},
\mathrm{~in~}
\mathbb{R}^{N}.
\end{aligned}
\end{equation*}
Set
\begin{equation*}
\begin{aligned}
U_{\sigma,\alpha}
(x)
=
S_{0,0}
^{\frac{(N-\alpha)(2s-N)}{4(N-\alpha+2s)}}
C(N,\alpha)
^{\frac{2s-N}{2(N-\alpha+2s)}}
U_{\sigma}(x).
\end{aligned}
\end{equation*}
Then,
for each
$\sigma>0$,
$U_{\sigma,\alpha}$
satisfies
\begin{equation*}
\begin{aligned}
(-\Delta)^{s}
U_{\sigma,\alpha}
=
\left(
\int_{\mathbb{R}^{N}}
\frac{|U_{\sigma,\alpha}|^{2^{*}_{\alpha}}}{|x-y|^{\alpha}}
\mathrm{d}y
\right)
|U_{\sigma,\alpha}|^{2^{*}_{\alpha}-2}U_{\sigma,\alpha},
\mathrm{~in~}
\mathbb{R}^{N}.
\end{aligned}
\end{equation*}
Hence,
we know that
$U_{\sigma,\alpha}\in \mathcal{N}^{3}$,
and
\begin{equation}\label{30}
\begin{aligned}
\|U_{\sigma,\alpha}\|_{D}^{2}
=
\int_{\mathbb{R}^{N}}
\int_{\mathbb{R}^{N}}
\frac{
|U_{\sigma,\alpha}(y)|^{2^{*}_{\alpha}}
|U_{\sigma,\alpha}(x)|^{2^{*}_{\alpha}}
}
{|x-y|^{\alpha}}
\mathrm{d}x
\mathrm{d}y
=
S_{0,\alpha}^{\frac{2N-\alpha}{N-\alpha+2s}}.
\end{aligned}
\end{equation}
Now,
we show that
$$c^{3}_{0}=\frac{N+2s-\alpha}{2(2N-\alpha)}S_{0,\alpha}^{\frac{2N-\alpha}{N+2s-\alpha}}.$$
Suppose on the contrary that
$c^{3}_{0}<\frac{N+2s-\alpha}{2(2N-\alpha)}S_{0,\alpha}^{\frac{2N-\alpha}{N+2s-\alpha}}$.
Then there exists
$\ddot{u}$
satisfies
$I_{3}(\ddot{u})=c_{0}^{3}$
and
$\ddot{u}\in \mathcal{N}^{3}$.
We get
\begin{equation}\label{31}
\begin{aligned}
c_{0}^{3}
=
I_{3}(\ddot{u})
-
\frac{1}{2\cdot 2^{*}_{\alpha}}
\langle I_{3}^{'}(\ddot{u}),\ddot{u}\rangle
=
\left(
\frac{1}{2^{*}}
-
\frac{1}{2\cdot 2^{*}_{\alpha}}
\right)
\|\ddot{u}\|^{2}_{D}.
\end{aligned}
\end{equation}
Combining
\eqref{30}
and
\eqref{31},
we know
\begin{equation*}
\begin{aligned}
\left(
\frac{1}{2^{*}}
-
\frac{1}{2\cdot 2^{*}_{\alpha}}
\right)
\|U_{\sigma,\alpha}\|^{2}_{D}
=
\frac{N+2s-\alpha}{2(2N-\alpha)}S_{0,\alpha}^{\frac{2N-\alpha}{N+2s-\alpha}}
>
c_{0}^{3}
=
\left(
\frac{1}{2^{*}}
-
\frac{1}{2\cdot 2^{*}_{\alpha}}
\right)
\|\ddot{u}\|^{2}_{D},
\end{aligned}
\end{equation*}
which implies that
\begin{equation*}
\begin{aligned}
S_{0,\alpha}^{\frac{2N-\alpha}{N-\alpha+2s}}
=
\|U_{\sigma,\alpha}\|^{2}_{D}
>
\|\ddot{u}\|^{2}_{D}.
\end{aligned}
\end{equation*}
This contradicts with
$
\|\ddot{u}\|^{2}_{D}
\geqslant
S_{0,\alpha}^{\frac{2N-\alpha}{N-\alpha+2s}}
$(see Lemma \ref{lemma14}).
Hence,
we know that
$$c^{3}_{0}=\frac{N+2s-\alpha}{2(2N-\alpha)}S_{0,\alpha}^{\frac{2N-\alpha}{N+2s-\alpha}}\geqslant\frac{N+2s-\alpha}{2(2N-\alpha)}S_{\zeta,\alpha}^{\frac{2N-\alpha}{N+2s-\alpha}}>c^{2}=c^{2}_{0}.$$
\end{proof}

\noindent
{\bf The proof of Theorem \ref{theorem2}:}
We divide our proof into five steps.

\noindent
{\bf Step 1.}
Since
$\{u_{n}\}$
is a bounded sequence in
$D^{s,2}(\mathbb{R}^{N})$,
up to a subsequence,
we can assume that
\begin{align*}
&
u_{n}\rightharpoonup u,
~
\mathrm{in}
~
D^{s,2}(\mathbb{R}^{N}),
~~
u_{n}\rightarrow u,
~
\mathrm{a.e. ~in}
~
\mathbb{R}^{N},\\
&u_{n}\rightarrow u,
~
\mathrm{in}
~
L^{r}_{loc}(\mathbb{R}^{N})
~
\mathrm{for~all}
~
r\in[1,2^{*}_{s}).
\end{align*}
According to
Lemma
\ref{lemma5},
Lemma
\ref{lemma8}
and
Lemma
\ref{lemma16},
there exists
$C>0$
such that
$$
\|u_{n}\|_{\mathcal{L}^{2,N-2s}(\mathbb{R}^{N})}\geqslant C>0.
$$
On the other hand,
since the sequence is bounded in
$D^{s,2}(\mathbb{R}^{N})$
and
$D^{s,2}(\mathbb{R}^{N})\hookrightarrow L^{2^{*}_{s}}(\mathbb{R}^{N})\hookrightarrow \mathcal{L}^{2,N-2s}(\mathbb{R}^{N})$,
we have
$$
\|u_{n}\|_{\mathcal{L}^{2,N-2s}(\mathbb{R}^{N})}\leqslant C,
$$
for some
$C>0$
independent of $n$.
Hence, there exists a positive constant which we denote again by $C$ such
that for any $n$ we obtain
$$
C
\leqslant
\|u_{n}\|_{\mathcal{L}^{2,N-2s}(\mathbb{R}^{N})}\leqslant C^{-1}.
$$
So we may find
$\sigma_{n} > 0$
and
$x_{n}\in \mathbb{R}^{N}$
such that
$$
\frac{1}{\sigma_{n}^{2s}}
\int_{B(x_{n},\sigma_{n})}
|u_{n}(y)|^{2}
\mathrm{d}y
\geqslant
\|u_{n}\|_{\mathcal{L}^{2,N-2s}(\mathbb{R}^{N})}^{2}
-
\frac{C}{2n}
\geqslant
C_{6}>0.
$$
Let
$\bar{u}_{n}(x)=\sigma_{n}^{\frac{N-2s}{2}}u_{n}(x_{n}+\sigma_{n}x)$.
We may readily verify that
$$\widetilde{I_{2}}(\bar{u}_{n})=I_{2}(u_{n})\rightarrow c^{2}
~\mathrm{and}~
\widetilde{I_{2}}^{'}(\bar{u}_{n})
\rightarrow0
~\mathrm{as}~n\rightarrow\infty,$$
where
\begin{equation*}
\begin{aligned}
\widetilde{I_{2}}(\bar{u}_{n})
=&
\frac{1}{2}
\|\bar{u}_{n}\|_{D}^{2}
-
\frac{1}{2}
\int_{\mathbb{R}^{N}}
\frac{|\bar{u}_{n}|^{2}}
{|x+\frac{x_{n}}{\sigma_{n}}|^{2s}}
\mathrm{d}x\\
&-
\frac{1}{2\cdot2^{*}_{\alpha}}
\int_{\mathbb{R}^{N}}
\int_{\mathbb{R}^{N}}
\frac{|\bar{u}_{n}(x)|^{2^{*}_{\alpha}}|\bar{u}_{n}(y)|^{2^{*}_{\alpha}}}{|x-y|^{\alpha}}
\mathrm{d}x
\mathrm{d}y
-
\sum_{i=1}^{k}
\frac{1}{2^{*}_{s,\theta_{i}}}
\int_{\mathbb{R}^{N}}
\frac{|\bar{u}_{n}|^{2^{*}_{s,\theta_{i}}}}
{|x+\frac{x_{n}}{\sigma_{n}}|^{\theta_{i}}}
\mathrm{d}x.
\end{aligned}
\end{equation*}
Now,
for all
$\varphi\in D^{s,2}(\mathbb{R}^{N})$,
we obtain
\begin{equation*}
\begin{aligned}
|\langle \widetilde{I_{2}}^{'}(\bar{u}_{n}),\varphi\rangle|
=&
|\langle I_{2}^{'}(u_{n}),\bar{\varphi}\rangle|\\
\leqslant&
\|I_{2}^{'}(u_{n})\|_{D^{-1}}
\|\bar{\varphi}\|_{D}\\
=&
o(1)
\|\bar{\varphi}\|_{D},
\end{aligned}
\end{equation*}
where
$\bar{\varphi}=\sigma_{n}^{-\frac{N-2s}{2}}\varphi(\frac{x-x_{n}}{\sigma_{n}})$.
Since
$\|\bar{\varphi}\|_{D}=\|\varphi\|_{D}$,
we get
$$\widetilde{I_{2}}^{'}(\bar{u}_{n})\rightarrow0~
\mathrm{as}~
n\rightarrow\infty.$$
Thus there exists
$\bar{u}$
such that
\begin{align*}
&
\bar{u}_{n}\rightharpoonup \bar{u},
~
\mathrm{in}
~
D^{s,2}(\mathbb{R}^{N}),
~~
\bar{u}_{n}\rightarrow \bar{u},
~
\mathrm{a.e. ~in}
~
\mathbb{R}^{N},\\
&\bar{u}_{n}\rightarrow \bar{u},
~
\mathrm{in}
~
L^{r}_{loc}(\mathbb{R}^{N})
~
\mathrm{for~all}
~
r\in[1,2^{*}_{s}).
\end{align*}
Then
$$
\int_{B(0,1)}
|\bar{u}_{n}(y)|^{2}
\mathrm{d}y
=
\frac{1}{\sigma_{n}^{2s}}
\int_{B(x_{n},\sigma_{n})}
|u_{n}(y)|^{2}
\mathrm{d}y
\geqslant
C_{6}>0.$$
As a result,
$\bar{u}\not\equiv0$.

\noindent
{\bf Step 2.}
Now,
we claim that
$\{\frac{x_{n}}{\sigma_{n}}\}$
is bounded.
If
$\frac{x_{n}}{\sigma_{n}}\rightarrow\infty$,
then for any
$\varphi\in D^{s,2}(\mathbb{R}^{N})$,
we get
\begin{equation}\label{32}
\begin{aligned}
\lim_{n\rightarrow\infty}
\int_{\mathbb{R}^{N}}
\frac{\bar{u}_{n}\varphi}
{|x+\frac{x_{n}}{\sigma_{n}}|^{2s}}
\mathrm{d}x
=
0
~\mathrm{and}~
\lim_{n\rightarrow\infty}
\int_{\mathbb{R}^{N}}
\frac{|\bar{u}_{n}|^{2^{*}_{s,\theta_{i}}-2}\bar{u}_{n}\varphi}
{|x+\frac{x_{n}}{\sigma_{n}}|^{\theta_{i}}}
\mathrm{d}x
=0.
\end{aligned}
\end{equation}
We  will show that
\begin{equation*}
\begin{aligned}
\langle I^{'}_{3}(\bar{u}),\varphi\rangle=0.
\end{aligned}
\end{equation*}
Since
$\bar{u}_{n}\rightharpoonup \bar{u}$
weakly in
$D^{s,2}(\mathbb{R}^{N})$,
we know
\begin{equation}\label{33}
\begin{aligned}
&\lim_{n\rightarrow\infty}
\int_{\mathbb{R}^{N}}
\int_{\mathbb{R}^{N}}
\frac{(\bar{u}_{n}(x)-\bar{u}_{n}(y))(\varphi(x)-\varphi(y))}
{|x-y|^{N+2s}}
\mathrm{d}x
\mathrm{d}y\\
=&
\int_{\mathbb{R}^{N}}
\int_{\mathbb{R}^{N}}
\frac{(\bar{u}(x)-\bar{u}(y))(\varphi(x)-\varphi(y))}
{|x-y|^{N+2s}}
\mathrm{d}x
\mathrm{d}y.
\end{aligned}
\end{equation}
By the Hardy--Littlewood--Sobolev inequality,
the Riesz potential defines a linear continuous map from
$L^{\frac{2N}{2N-\alpha}}(\mathbb{R}^{N})$
to
$L^{\frac{2N}{\alpha}}(\mathbb{R}^{N})$.
Since
$|\bar{u}_{n}|^{2^{*}_{\alpha}}\rightharpoonup|\bar{u}|^{2^{*}_{\alpha}}$
weakly in
$L^{\frac{2^{*}}{2^{*}_{\alpha}}}(\mathbb{R}^{N})$,
it follows that
as
$n\rightarrow\infty$,
\begin{equation}\label{34}
\int_{\mathbb{R}^{N}}
\frac{|\bar{u}_{n}(y)|^{2^{*}_{\alpha}}}{|x-y|^{\alpha}}
\mathrm{d}y
\rightharpoonup
\int_{\mathbb{R}^{N}}
\frac{|\bar{u}(y)|^{2^{*}_{\alpha}}}{|x-y|^{\alpha}}
\mathrm{d}y
~\mathrm{weakly~in}~L^{\frac{2N}{\alpha}}(\mathbb{R}^{N}).
\end{equation}
Now,
we show that
$|\bar{u}_{n}|^{2^{*}_{\alpha}-2}
\bar{u}_{n}
\varphi
\rightarrow
|\bar{u}|^{2^{*}_{\alpha}-2}
\bar{u}\varphi$
in
$L^{\frac{2N}{2N-\alpha}}(\mathbb{R}^{N})$.
For any $\varepsilon>0$,
there exists $R>0$ large enough such that
\begin{equation}\label{35}
\begin{aligned}
&
\lim_{n\rightarrow\infty}
\int_{|x| >R}
\left|
|\bar{u}_{n}|^{2^{*}_{\alpha}-2}
\bar{u}_{n}
\varphi
\right|^{\frac{2N}{2N-\alpha}}
-
\left|
|\bar{u}|^{2^{*}_{\alpha}-2}
\bar{u}\varphi
\right|^{\frac{2N}{2N-\alpha}}
\mathrm{d}x\\
\leqslant&
\lim_{n\rightarrow\infty}
\int_{|x| >R}
\left|
\bar{u}_{n}
\right|^{(2^{*}_{\alpha}-1)\cdot\frac{2^{*}_{s}}{2^{*}_{\alpha}}}
\left|
\varphi
\right|^{\frac{2^{*}_{s}}{2^{*}_{\alpha}}}
\mathrm{d}x
+
\int_{|x| >R}
\left|
\bar{u}
\right|^{(2^{*}_{\alpha}-1)\cdot\frac{2^{*}_{s}}{2^{*}_{\alpha}}}
\left|
\varphi
\right|^{\frac{2^{*}_{s}}{2^{*}_{\alpha}}}
\mathrm{d}x\\
\leqslant&
\lim_{n\rightarrow\infty}
\left(
\int_{|x| >R}
\left|
\bar{u}_{n}
\right|^{2^{*}_{s}}
\mathrm{d}x
\right)^{1-\frac{1}{2^{*}_{\alpha}}}
\left(
\int_{|x| >R}
\left|
\varphi
\right|^{2^{*}_{s}}
\mathrm{d}x
\right)^{\frac{1}{2^{*}_{\alpha}}}\\
&+
\left(
\int_{|x| >R}
\left|
\bar{u}
\right|^{2^{*}_{s}}
\mathrm{d}x
\right)^{1-\frac{1}{2^{*}_{\alpha}}}
\left(
\int_{|x| >R}
\left|
\varphi
\right|^{2^{*}_{s}}
\mathrm{d}x
\right)^{\frac{1}{2^{*}_{\alpha}}}\\
\leqslant&
C
\left(
\int_{|x| >R}
\left|
\varphi
\right|^{2^{*}_{s}}
\mathrm{d}x
\right)^{\frac{1}{2^{*}_{\alpha}}}\\
<&
\frac{\varepsilon}{2}.
\end{aligned}
\end{equation}
On the other hand,
by the boundedness of
$\{\bar{u}_{n}\}$,
one has
\begin{equation*}
\begin{aligned}
\left(
\int_{|x|\leqslant R}
\left|
\bar{u}_{n}
\right|^{2^{*}_{s}}
\mathrm{d}x
\right)^{1-\frac{1}{2^{*}_{\alpha}}}
\leqslant&
M.
\end{aligned}
\end{equation*}
where $M>0$ is a constant.
Let
$\Omega=\{x\in\mathbb{R}^{N}||x|\leqslant R\}$.
For any $\tilde{\varepsilon}>0$,
there exists
$\delta>0$,
when
$E\subset\Omega$
with
$|E|<\delta$.
We obtain
\begin{align*}
\int_{E}
\left|
|\bar{u}_{n}|^{2^{*}_{\alpha}-2}
\bar{u}_{n}
\varphi
\right|^{\frac{2N}{2N-\alpha}}
\mathrm{d}x
=&
\int_{E}
\left|
\bar{u}_{n}
\right|^{(2^{*}_{\alpha}-1)\cdot\frac{2^{*}_{s}}{2^{*}_{\alpha}}}
\left|
\varphi
\right|^{\frac{2^{*}_{s}}{2^{*}_{\alpha}}}
\mathrm{d}x\\
\leqslant&
\left(
\int_{E}
\left|
\bar{u}_{n}
\right|^{2^{*}_{s}}
\mathrm{d}x
\right)^{1-\frac{1}{2^{*}_{\alpha}}}
\left(
\int_{E}
\left|
\varphi
\right|^{2^{*}_{s}}
\mathrm{d}x
\right)^{\frac{1}{2^{*}_{\alpha}}}\\
<&M\tilde{\varepsilon},
\end{align*}
where the last inequality is from the absolutely continuity of
$\int_{E}
\left|
\varphi
\right|^{2^{*}_{s}}
\mathrm{d}x$.
Moreover,
$|\bar{u}_{n}|^{2^{*}_{\alpha}-2}
\bar{u}_{n}
\varphi
\rightarrow
|\bar{u}|^{2^{*}_{\alpha}-2}
\bar{u}
\varphi$
a.e. in $\mathbb{R}^{N}$ as $n\rightarrow\infty$.
Thus,
by the Vitali convergence Theorem,
we get
\begin{equation}\label{36}
\begin{aligned}
\lim_{n\rightarrow\infty}
\int_{|x|\leqslant R}
\left|
|\bar{u}_{n}|^{2^{*}_{\alpha}-2}
\bar{u}_{n}
\varphi
\right|^{\frac{2N}{2N-\alpha}}
\mathrm{d}x
=
\int_{|x|\leqslant R}
\left|
|\bar{u}|^{2^{*}_{\alpha}-2}
\bar{u}
\varphi
\right|^{\frac{2N}{2N-\alpha}}
\mathrm{d}x.
\end{aligned}
\end{equation}
It follows from
\eqref{35}
and
\eqref{36}
that
\begin{equation*}
\begin{aligned}
&
\lim_{n\rightarrow\infty}
\left|
\int_{\mathbb{R}^{N}}
\left|
|\bar{u}_{n}|^{2^{*}_{\alpha}-2}
\bar{u}_{n}
\varphi
\right|^{\frac{2N}{2N-\alpha}}
-
\left|
|\bar{u}|^{2^{*}_{\alpha}-2}
\bar{u}
\varphi
\right|^{\frac{2N}{2N-\alpha}}
\mathrm{d}x
\right|\\
\leqslant&
\lim_{n\rightarrow\infty}
\left|
\int_{|x|\leqslant R}
\left|
|\bar{u}_{n}|^{2^{*}_{\alpha}-2}
\bar{u}_{n}
\varphi
\right|^{\frac{2N}{2N-\alpha}}
-
\left|
|\bar{u}|^{2^{*}_{\alpha}-2}
\bar{u}
\varphi
\right|^{\frac{2N}{2N-\alpha}}
\mathrm{d}x
\right|\\
&+
\lim_{n\rightarrow\infty}
\left|
\int_{|x|> R}
\left|
|\bar{u}_{n}|^{2^{*}_{\alpha}-2}
\bar{u}_{n}
\varphi
\right|^{\frac{2N}{2N-\alpha}}
-
\left|
|\bar{u}|^{2^{*}_{\alpha}-2}
\bar{u}
\varphi
\right|^{\frac{2N}{2N-\alpha}}
\mathrm{d}x
\right|\\
<&
\varepsilon.
\end{aligned}
\end{equation*}
This implies that
\begin{equation}\label{37}
\begin{aligned}
\lim_{n\rightarrow\infty}
\int_{\mathbb{R}^{N}}
\left|
|\bar{u}_{n}|^{2^{*}_{\alpha}-2}
\bar{u}_{n}
\varphi
\right|^{\frac{2N}{2N-\alpha}}
\mathrm{d}x
=
\int_{\mathbb{R}^{N}}
\left|
|\bar{u}|^{2^{*}_{\alpha}-2}
\bar{u}
\varphi
\right|^{\frac{2N}{2N-\alpha}}
\mathrm{d}x
\end{aligned}
\end{equation}
Combining
\eqref{34}
and
\eqref{37},
we have
\begin{equation}\label{38}
\begin{aligned}
&\lim_{n\rightarrow\infty}
\int_{\mathbb{R}^{N}}
\int_{\mathbb{R}^{N}}
\frac{
|\bar{u}_{n}(y)|^{2^{*}_{\alpha}}
|\bar{u}_{n}(x)|^{2^{*}_{\alpha}-2}
\bar{u}_{n}(x)
\varphi(x)
}{|x-y|^{\alpha}}
\mathrm{d}y
\mathrm{d}x\\
=&
\int_{\mathbb{R}^{N}}
\int_{\mathbb{R}^{N}}
\frac{
|\bar{u}(y)|^{2^{*}_{\alpha}}
|\bar{u}(x)|^{2^{*}_{\alpha}-2}
\bar{u}(x)
\varphi(x)
}{|x-y|^{\alpha}}
\mathrm{d}y
\mathrm{d}x.
\end{aligned}
\end{equation}
Applying
$\lim\limits_{n\rightarrow\infty}\langle \widetilde{I_{2}}^{'}(\bar{u}_{n}),\varphi\rangle\rightarrow0$,
\eqref{32},
\eqref{33}
and
\eqref{38}
we know
\begin{equation}\label{39}
\begin{aligned}
\langle I^{'}_{3}(\bar{u}),\varphi\rangle=0.
\end{aligned}
\end{equation}
Moreover,
according to \eqref{39}
and
$\bar{u}\not\equiv0$,
we get that
$$\bar{u}\in \mathcal{N}^{3}.$$
By Br\'{e}zis--Lieb lemma \cite[Lemma 2.2]{Gao2016},
we have
\begin{equation*}
\begin{aligned}
&\int_{\mathbb{R}^{N}}
\int_{\mathbb{R}^{N}}
\frac{|\bar{u}_{n}(x)|^{2^{*}_{\alpha}}|\bar{u}_{n}(y)|^{2^{*}_{\alpha}}}{|x-y|^{\alpha}}
\mathrm{d}x
\mathrm{d}y
-
\int_{\mathbb{R}^{N}}
\int_{\mathbb{R}^{N}}
\frac{|\bar{u}_{n}(x)-\bar{u}(x)|^{2^{*}_{\alpha}}|\bar{u}_{n}(y)-\bar{u}(y)|^{2^{*}_{\alpha}}}{|x-y|^{\alpha}}
\mathrm{d}x
\mathrm{d}y\\
=&
\int_{\mathbb{R}^{N}}
\int_{\mathbb{R}^{N}}
\frac{|\bar{u}(x)|^{2^{*}_{\alpha}}|\bar{u}(y)|^{2^{*}_{\alpha}}}{|x-y|^{\alpha}}
\mathrm{d}x
\mathrm{d}y
+
o(1),
\end{aligned}
\end{equation*}
which implies that
\begin{equation}\label{40}
\begin{aligned}
\int_{\mathbb{R}^{N}}
\int_{\mathbb{R}^{N}}
\frac{|\bar{u}_{n}(x)|^{2^{*}_{\alpha}}|\bar{u}_{n}(y)|^{2^{*}_{\alpha}}}{|x-y|^{\alpha}}
\mathrm{d}x
\mathrm{d}y
\geqslant
\int_{\mathbb{R}^{N}}
\int_{\mathbb{R}^{N}}
\frac{|\bar{u}(x)|^{2^{*}_{\alpha}}|\bar{u}(y)|^{2^{*}_{\alpha}}}{|x-y|^{\alpha}}
\mathrm{d}x
\mathrm{d}y
+
o(1).
\end{aligned}
\end{equation}
Similarly,
we get
\begin{equation}\label{41}
\begin{aligned}
\int_{\mathbb{R}^{N}}
\frac{|\bar{u}_{n}|^{2^{*}_{s,\theta_{i}}}}
{|x|^{\theta_{i}}}
\mathrm{d}x
\geqslant
\int_{\mathbb{R}^{N}}
\frac{|\bar{u}|^{2^{*}_{s,\theta_{i}}}}
{|x|^{\theta_{i}}}
\mathrm{d}x
+
o(1).
\end{aligned}
\end{equation}
Applying
Lemma \ref{lemma17},
Lemma \ref{lemma12},
(\ref{40}),
(\ref{41}),
$\bar{u}\in \mathcal{N}^{3}$
and
Lemma \ref{lemma14},
we obtain
\begin{align*}
c_{0}^{3}
>
c_{0}^{2}
=c^{2}
=&
I_{2}(\bar{u}_{n})
-
\frac{1}{2}
\langle I^{'}_{2}(\bar{u}_{n}),\bar{u}_{n}\rangle\\
\geqslant&
\left(
\frac{1}{2}
-
\frac{1}{2\cdot2^{*}_{\alpha}}
\right)
\int_{\mathbb{R}^{N}}
\int_{\mathbb{R}^{N}}
\frac{|\bar{u}(x)|^{2^{*}_{\alpha}}|\bar{u}(y)|^{2^{*}_{\alpha}}}{|x-y|^{\alpha}}
\mathrm{d}x
\mathrm{d}y
+o(1)\\
=&
I_{3}(\bar{u})
-
\frac{1}{2}
\langle I_{3}^{'}(\bar{u}),\bar{u}\rangle
=
I_{3}(\bar{u})
\geqslant
c_{0}^{3},
\end{align*}
which yields a contradiction.
Hence, $\{\frac{x_{n}}{\sigma_{n}}\}$ is bounded.

\noindent
{\bf Step 3.}
In this step,
we study another
$(PS)_{c^{2}}$
sequence of
$I_{2}$.
Let
$\tilde{u}_{n}(x)=\sigma_{n}^{\frac{N-2s}{2}}u_{n}(\sigma_{n}x)$.
Then we can
verify that
$$I_{2}(\tilde{u}_{n})
=
I_{2}(u_{n})\rightarrow c^{2}
,~
I^{'}_{2}(\tilde{u}_{n})\rightarrow 0
~\mathrm{as}~n\rightarrow\infty.$$
Arguing as before, we have
\begin{align*}
&\tilde{u}_{n}\rightharpoonup \tilde{u}
~
\mathrm{in}
~
D^{1,2}(\mathbb{R}^{N}),~
\tilde{u}_{n}\rightarrow \tilde{u}
~
\mathrm{a.e. ~in}
~
\mathbb{R}^{N},\\
&\tilde{u}_{n}\rightarrow \tilde{u}
~
\mathrm{in}
~
L^{r}_{loc}(\mathbb{R}^{N})~~\mathrm{for~all~}r\in[1,2^{*}_{s}).
\end{align*}
By using  $\{\frac{x_{n}}{\sigma_{n}}\}$ is bounded,
there exists
$\tilde{R}>0$ such that
$$
\int_{B(0,\tilde{R})}
|\tilde{u}_{n}(y)|^{2}
\mathrm{d}y
>
\int_{B(\frac{x_{n}}{\sigma_{n}},1)}
|\tilde{u}_{n}(y)|^{2}
\mathrm{d}y
=
\frac{1}{\sigma_{n}^{2s}}
\int_{B(x_{n},\sigma_{n})}
|u_{n}(y)|^{2}
\mathrm{d}y
\geqslant
C_{6}>0.$$
As a result,
$\tilde{u}\not\equiv0$.

\noindent
{\bf Step 4.}
In this step,
we show
$\tilde{u}_{n}\rightarrow \tilde{u}$
strongly in
$D^{s,2}(\mathbb{R}^{N})$.
Set
\begin{align*}
K(u)
=
\left(
\frac{1}{2}
-
\frac{1}{2\cdot2^{*}_{\alpha}}
\right)
\int_{\mathbb{R}^{N}}
\int_{\mathbb{R}^{N}}
\frac{|u(x)|^{2^{*}_{\alpha}}|u(y)|^{2^{*}_{\alpha}}}{|x-y|^{\alpha}}
\mathrm{d}x
\mathrm{d}y
+
\sum_{i=1}^{k}
\left(
\frac{1}{2}
-
\frac{1}{2^{*}_{s,\theta_{i}}}
\right)
\int_{\mathbb{R}^{N}}
\frac{|u|^{2^{*}_{s,\theta_{i}}}}
{|x|^{\theta_{i}}}
\mathrm{d}x.
\end{align*}
Similar to Step 2,
we know that
\begin{equation}\label{42}
\begin{aligned}
\langle I^{'}_{2}(\tilde{u}),\varphi\rangle=0.
\end{aligned}
\end{equation}
Applying
Lemma \ref{lemma11},
Lemma \ref{lemma12},
$\tilde{u}\in \mathcal{N}^{2}$
and
(\ref{40})
--
(\ref{42}),
we obtain
\begin{equation}\label{43}
\begin{aligned}
c_{0}^{2}
=c^{2}
=&
I_{2}(\tilde{u}_{n})
-
\frac{1}{2}
\langle I^{'}_{2}(\tilde{u}_{n}),\tilde{u}_{n}\rangle\\
=&
\lim_{n\rightarrow\infty}K(\tilde{u}_{n})+o(1)\\
\geqslant&
K(\tilde{u})
+o(1)\\
=&
I_{2}(\tilde{u})
-
\frac{1}{2}
\langle I_{2}^{'}(\tilde{u}),\tilde{u}\rangle
=
I_{2}(\tilde{u})
\geqslant
c_{0}^{2}.
\end{aligned}
\end{equation}
\textcolor{red}{Therefore, the inequalities above have to be equalities}.
We know
\begin{align*}
\lim\limits_{n\rightarrow\infty}
K(\tilde{u}_{n})
=
K(\tilde{u}).
\end{align*}
By using Br\'{e}zis--Lieb lemma again,
we have
\begin{align*}
\lim\limits_{n\rightarrow\infty}
K(\tilde{u}_{n})
-
\lim\limits_{n\rightarrow\infty}
K(\tilde{u}_{n}-\tilde{u})
=
K(\tilde{u})+o(1).
\end{align*}
Hence,
we deduce that
\begin{align*}
\lim\limits_{n\rightarrow\infty}
K(\tilde{u}_{n}-\tilde{u})
=
0,
\end{align*}
which implies that
\begin{equation}\label{44}
\begin{aligned}
&
\lim\limits_{n\rightarrow\infty}
\int_{\mathbb{R}^{N}}
\int_{\mathbb{R}^{N}}
\frac{|\tilde{u}_{n}(x)-\tilde{u}(x)|^{2^{*}_{\alpha}}|\tilde{u}_{n}(y)-\tilde{u}(y)|^{2^{*}_{\alpha}}}{|x-y|^{\alpha}}
\mathrm{d}x
\mathrm{d}y
=0,
\\
&
\lim\limits_{n\rightarrow\infty}
\int_{\mathbb{R}^{N}}
\frac{
|\tilde{u}_{n}-\tilde{u}|^{2^{*}_{s,\theta_{i}}}
}
{|x|^{\theta_{i}}}
\mathrm{d}x
=0,~\mathrm{for~all}~
i=1,\ldots,k.
\end{aligned}
\end{equation}
According to
$\langle I^{'}_{2}(\tilde{u}_{n}),\tilde{u}_{n}\rangle=o(1)$,
$\langle I^{'}_{2}(\tilde{u}),\tilde{u}\rangle=0$
and Br\'{e}zis--Lieb lemma,
we obtain
\begin{equation*}
\begin{aligned}
o(1)=&
\langle I^{'}_{2}(\tilde{u}_{n}),\tilde{u}_{n}\rangle
-
\langle I^{'}_{2}(\tilde{u}),\tilde{u}\rangle\\
=&
\|\tilde{u}_{n}-\tilde{u}\|_{\zeta}^{2}
-
\int_{\mathbb{R}^{N}}
\int_{\mathbb{R}^{N}}
\frac{|\tilde{u}_{n}(x)-\tilde{u}(x)|^{2^{*}_{\alpha_{i}}}|\tilde{u}_{n}(y)-\tilde{u}(y)|^{2^{*}_{\alpha_{i}}}}{|x-y|^{\alpha_{i}}}
\mathrm{d}x
\mathrm{d}y\\
&-\sum_{i=1}^{k}
\int_{\mathbb{R}^{N}}
\frac{
|\tilde{u}_{n}-\tilde{u}|^{2^{*}_{s,\theta_{i}}}
}
{|x|^{\theta_{i}}}
\mathrm{d}x
+o(1),
\end{aligned}
\end{equation*}
which implies that
\begin{equation}\label{45}
\begin{aligned}
&
\lim\limits_{n\rightarrow\infty}
\|\tilde{u}_{n}-\tilde{u}\|_{\zeta}^{2}\\
=&
\lim\limits_{n\rightarrow\infty}
\int_{\mathbb{R}^{N}}
\int_{\mathbb{R}^{N}}
\frac{|\tilde{u}_{n}(x)-\tilde{u}(x)|^{2^{*}_{\alpha}}|\tilde{u}_{n}(y)-\tilde{u}(y)|^{2^{*}_{\alpha}}}{|x-y|^{\alpha}}
\mathrm{d}x
\mathrm{d}y\\
&+
\lim\limits_{n\rightarrow\infty}
\int_{\mathbb{R}^{N}}
\frac{
|\tilde{u}_{n}-\tilde{u}|^{2^{*}_{s,\theta_{i}}}
}
{|x|^{\theta_{i}}}
\mathrm{d}x+o(1).
\end{aligned}
\end{equation}
Combining
(\ref{44})
and
(\ref{45}),
we get
\begin{equation*}
\begin{aligned}
\lim\limits_{n\rightarrow\infty}
\|\tilde{u}_{n}-\tilde{u}\|_{\zeta}^{2}
=
0.
\end{aligned}
\end{equation*}
Since
$\tilde{u}\not\equiv0$,
we know that
$\tilde{u}_{n}\rightarrow \tilde{u}$
strongly in
$D^{s,2}(\mathbb{R}^{N})$.

\noindent
{\bf Step 5.}
By using \eqref{43} again,
we know that
$I_{2}(\tilde{u})=c^{2}$,
which means that $\tilde{u}$ is a nontrivial solution of problem $(\mathcal{P}_{2})$ at the
energy level $c^{2}$.
Then we have just to prove
that we can choose
$\tilde{u}\geqslant0$.
We know that
\begin{equation*}
\begin{aligned}
0=&
\langle I^{'}_{2}(\tilde{u}),\tilde{u}^{-}\rangle\\
=&
\int_{\mathbb{R}^{N}}
\int_{\mathbb{R}^{N}}
\frac{(\tilde{u}(x)-\tilde{u}(y))(\tilde{u}^{-}(x)-\tilde{u}^{-}(y))}{|x-y|^{N+2s}}
\mathrm{d}x
\mathrm{d}y
-
\zeta
\int_{\mathbb{R}^{N}}
\frac{\tilde{u}\tilde{u}^{-}}{|x|^{2s}}
\mathrm{d}x
\\
&-
\int_{\mathbb{R}^{N}}
\int_{\mathbb{R}^{N}}
\frac{|\tilde{u}(y)|^{2^{*}_{\alpha}}|\tilde{u}(x)|^{2^{*}_{\alpha}-2}\tilde{u}(x)\tilde{u}^{-}(x)}{|x-y|^{\alpha}}
\mathrm{d}x
\mathrm{d}y
-
\sum_{i=1}^{k}
\int_{\mathbb{R}^{N}}
\frac{|\tilde{u}|^{2^{*}_{s,\theta_{i}}-2}\tilde{u}\tilde{u}^{-}}
{|x|^{\theta_{i}}}
\mathrm{d}x,
\end{aligned}
\end{equation*}
where
$\tilde{u}^{-}=\max\{0,-\tilde{u}\}$.
For a.e.
$x,y\in \mathbb{R}^{N}$,
we have
$$(\tilde{u}(x)-\tilde{u}(y))(\tilde{u}^{-}(x)-\tilde{u}^{-}(y))\leqslant -|\tilde{u}^{-}(x)-\tilde{u}^{-}(y)|^{2}.$$
Then,
we get
\begin{equation*}
\begin{aligned}
0=&
-
\|\tilde{u}^{-}\|_{D}^{2}
-
\zeta
\int_{\mathbb{R}^{N}}
\frac{|\tilde{u}^{-}|^{2}}{|x|^{2s}}
\mathrm{d}x
-
\int_{\mathbb{R}^{N}}
\int_{\mathbb{R}^{N}}
\frac{|\tilde{u}(y)|^{2^{*}_{\alpha}}|\tilde{u}^{-}(x)|^{2^{*}_{\alpha}}}{|x-y|^{\alpha}}
\mathrm{d}x
\mathrm{d}y
-
\sum_{i=1}^{k}
\int_{\mathbb{R}^{N}}
\frac{|\tilde{u}^{-}|^{2^{*}_{s,\theta_{i}}}}
{|x|^{\theta_{i}}}
\mathrm{d}x\\
\leqslant&
-\|\tilde{u}^{-}\|_{D}^{2}.
\end{aligned}
\end{equation*}
Thus,
$\|\tilde{u}^{-}\|_{D}^{2}=0$.
Hence,
we can choose
$\tilde{u}\geqslant0$.
By using the fractional Kelvin transformation
\begin{equation}\label{46}
\begin{aligned}
\tilde{\tilde{u}}(x)
=
\frac{1}{|x|^{N-2s}}
\tilde{u}
\left(
\frac{x}{|x|^{2}}
\right).
\end{aligned}
\end{equation}
It is well known that
\begin{equation}\label{47}
\begin{aligned}
(-
\Delta)^{s}
\tilde{\tilde{u}}(x)
=
\frac{1}{|x|^{N+2s}}
(-
\Delta)^{s}
\tilde{u}
\left(
\frac{x}{|x|^{2}}
\right).
\end{aligned}
\end{equation}
The following identity is very useful.
For
$\forall x,y\in \mathbb{R}^{N}\backslash\{0\}$,
we get
\begin{equation}\label{48}
\begin{aligned}
\frac{1}
{
\left|
\frac{x}{|x|^{2}}-\frac{y}{|y|^{2}}
\right|^{\alpha}}
\cdot
\frac{1}{|xy|^{\alpha}}
=&
\frac{1}
{
\left|
\frac{x\cdot y^{2}-y\cdot x^{2}}{(xy)^{2}}
\right|^{\alpha}}
\cdot
\frac{1}{|xy|^{\alpha}}\\
=&
\frac{1}
{
\left|
\frac{x\cdot y^{2}-y\cdot x^{2}}{xy}
\right|^{\alpha}}\\
=&
\frac{1}
{
\left|
y-x
\right|^{\alpha}}.
\end{aligned}
\end{equation}
Set
$z=\frac{y}{|y|^{2}}$.
Applying
(\ref{46})
and
(\ref{48}),
we have
\begin{equation}\label{49}
\begin{aligned}
\int_{\mathbb{R}^{N}}
\frac{|\tilde{\tilde{u}}(y)|^{\frac{2N-\alpha}{N-2s}}}
{|x-y|^{\alpha}}
\mathrm{d}y
=&
\int_{\mathbb{R}^{N}}
\frac{|\tilde{u}
\left(
\frac{y}{|y|^{2}}
\right)|^{\frac{2N-\alpha}{N-2s}}}
{|x-y|^{\alpha}}
\cdot
\frac{1}{|y|^{2N-\alpha}}
\mathrm{d}y
~(\mathrm{by}~(\ref{46}))
\\
=&
\int_{\mathbb{R}^{N}}
\frac{|\tilde{u}
\left(
\frac{y}{|y|^{2}}
\right)|^{\frac{2N-\alpha}{N-2s}}}
{|\frac{x}{|x|^{2}}-\frac{y}{|y|^{2}}|^{\alpha}}
\cdot
\frac{1}{|xy|^{\alpha}\cdot|y|^{2N-\alpha}}
\mathrm{d}y
~(\mathrm{by}~(\ref{48}))\\
=&
\frac{1}{|x|^{\alpha}}
\int_{\mathbb{R}^{N}}
\frac{|\tilde{u}
\left(
\frac{y}{|y|^{2}}
\right)|^{\frac{2N-\alpha}{N-2s}}}
{|\frac{x}{|x|^{2}}-\frac{y}{|y|^{2}}|^{\alpha}}
\cdot
\frac{1}{|y|^{2N}}
\mathrm{d}y\\
=&
\frac{1}{|x|^{\alpha}}
\int_{\mathbb{R}^{N}}
\frac{|\tilde{u}
\left(
z
\right)|^{\frac{2N-\alpha}{N-2s}}}
{|\frac{x}{|x|^{2}}-z|^{\alpha}}
\mathrm{d}z
~(\mathrm{set}~z=\frac{y}{|y|^{2}}).
\end{aligned}
\end{equation}
By using
\eqref{46},
we get
\begin{equation}\label{50}
\begin{aligned}
\frac{
\left|
\tilde{\tilde{u}}
\right|^{\frac{4s-2\theta_{i}}{N-2s}}
\tilde{\tilde{u}}}
{|x|^{\theta_{i}}}
=
\frac{1}{|x|^{N+2s}}
\frac{
\tilde{u}
\left(
\frac{x}{|x|^{2}}
\right)
}
{\left|\frac{x}{|x|^{2}}\right|^{\theta_{i}}},
\end{aligned}
\end{equation}
and
\begin{equation}\label{59}
\begin{aligned}
\frac{\tilde{\tilde{u}}(x)}{|x|^{2s}}
=
\frac{1}{|x|^{N+2s}}
\frac{
\tilde{u}
\left(
\frac{x}{|x|^{2}}
\right)
}
{\left|\frac{x}{|x|^{2}}\right|^{2s}}.
\end{aligned}
\end{equation}
Therefore,
by using
(\ref{47})
and
(\ref{49})
--
(\ref{59}),
we get
\begin{equation*}
\begin{aligned}
(-
\Delta)^{s}
\tilde{\tilde{u}}
-\zeta
\frac{\tilde{\tilde{u}}}{|x|^{2s}}
=&
\left(
\int_{\mathbb{R}^{N}}
\frac{|\tilde{\tilde{u}}|^{\frac{2N-\alpha}{N-2s}}}
{|x-y|^{\alpha}}
\mathrm{d}y
\right)
\left|
\tilde{\tilde{u}}
\right|^{\frac{4s-\alpha}{N-2s}}
\tilde{\tilde{u}}
+
\sum\limits^{k}_{i=1}
\frac{
\left|
\tilde{\tilde{u}}
\right|^{\frac{4s-2\theta_{i}}{N-2s}}
\tilde{\tilde{u}}}
{|x|^{\theta_{i}}},
~\mathrm{in}~\mathbb{R}^{N}\backslash\{0\}.
\end{aligned}
\end{equation*}
\qed

\section{The proof of Theorem \ref{theorem1}}
In this section,
we study the existence of nonnegative solution of problem $(\mathcal{P}_{1})$.
\begin{lemma}\label{lemma18}
Assume that the assumptions of Theorem
\ref{theorem1}
hold.
Then
there exists a
$(PS)_{c^{1}}$
sequence of
$I_{1}$
at level
$c^{1}$,
where
$$0<c^{1}<c^{1,*}=
\min
\left\{
\frac{N+2s-\alpha_{1}}{2(2N-\alpha_{1})}S_{0,\alpha_{1}}^{\frac{2N-\alpha_{1}}{N+2s-\alpha_{1}}}
,
\ldots
,
\frac{N+2s-\alpha_{k}}{2(2N-\alpha_{k})}S_{0,\alpha_{k}}^{\frac{2N-\alpha_{k}}{N+2s-\alpha_{k}}}
,
\frac{s}{N}
S_{0,0} ^{\frac{N}{2s}}
\right\}.$$
\end{lemma}
\begin{lemma}\label{lemma13}
Assume that the assumptions of Theorem
\ref{theorem1}
hold.
Then
$$c_{1}^{1}=c^{1}=c_{0}^{1}=\inf_{u\in\mathcal{N}^{1}}I_{1}(u)>0.$$
\end{lemma}

The following result implies the non--vanishing of
$(PS)_{c^{1}}$
sequences.
\begin{lemma}\label{lemma19}
Assume that the assumptions of Theorem
\ref{theorem1}
hold.
Let
$\{u_{n}\}$
be a
$(PS)_{c^{1}}$
sequence of
$I_{1}$
with
$c^{1}\in(0,c^{1,*})$,
then
\begin{equation*}
\begin{aligned}
\lim_{n\rightarrow\infty}
\int_{\mathbb{R}^{N}}
|u_{n}|^{2^{*}_{s}}
\mathrm{d}x>0,
\end{aligned}
\end{equation*}
and
\begin{equation*}
\begin{aligned}
\lim_{n\rightarrow\infty}
\int_{\mathbb{R}^{N}}
\int_{\mathbb{R}^{N}}
\frac{|u_{n}(x)|^{2^{*}_{\alpha_{i}}}|u_{n}(y)|^{2^{*}_{\alpha_{i}}}}{|x-y|^{\alpha_{i}}}
\mathrm{d}x
\mathrm{d}y>0,(i=1,\ldots,k).
\end{aligned}
\end{equation*}
\end{lemma}
\noindent
{\bf Proof.}
Let
$\{u_{n}\}$
be a
$(PS)_{c^{1}}$
sequence of
$I_{1}$
with
$c^{1}\in(0,c^{1,*})$,
It's easy to see that
$\{u_{n}\}$
is uniformly bounded in
$D^{s,2}(\mathbb{R}^{N})$.
The proof of this Lemma is divided into three cases:

\noindent
(1)
$
\lim_{n\rightarrow\infty}
\int_{\mathbb{R}^{N}}
\int_{\mathbb{R}^{N}}
\frac{|u_{n}(x)|^{2^{*}_{\alpha_{1}}}|u_{n}(y)|^{2^{*}_{\alpha_{1}}}}{|x-y|^{\alpha_{1}}}
\mathrm{d}x
\mathrm{d}y>0
$;

\noindent
(2)
$
\lim_{n\rightarrow\infty}
\int_{\mathbb{R}^{N}}
\int_{\mathbb{R}^{N}}
\frac{|u_{n}(x)|^{2^{*}_{\alpha_{j}}}|u_{n}(y)|^{2^{*}_{\alpha_{j}}}}{|x-y|^{\alpha_{j}}}
\mathrm{d}x
\mathrm{d}y>0,
$
for
$j=2,\ldots,k$;

\noindent
(3)
$
\lim_{n\rightarrow\infty}
\int_{\mathbb{R}^{N}}
|u_{n}|^{2^{*}_{s}}
\mathrm{d}x>0.
$

\noindent
{\bf Case 1.}
Suppose that
\begin{equation}\label{51}
\begin{aligned}
\lim_{n\rightarrow\infty}
\int_{\mathbb{R}^{N}}
\int_{\mathbb{R}^{N}}
\frac{|u_{n}(x)|^{2^{*}_{\alpha_{1}}}|u_{n}(y)|^{2^{*}_{\alpha_{1}}}}
{|x-y|^{\alpha_{1}}}
\mathrm{d}x
\mathrm{d}y=0.
\end{aligned}
\end{equation}
Since
$\{u_{n}\}$
is uniformly bounded in
$D^{s,2}(\mathbb{R}^{N})$,
there exists a constant
$0<C<\infty$
such that
$\|u_{n}\|_{D}\leqslant C$.
By using (\ref{51})
and the definition of
fractional Coulomb--Sobolev space,
we obtain
$u_{n}\in \mathcal{E}^{s,\alpha_{1},2^{*}_{\alpha_{1}}}(\mathbb{R}^{N})$.
Applying Lemma \ref{lemma2} and (\ref{51}),
we have
\begin{equation}\label{52}
\begin{aligned}
&
\lim_{n\rightarrow\infty}
\|u_{n}\|_{L^{2^{*}_{s}}(\mathbb{R}^{N})}\\
\leqslant&
C
\left(
\lim_{n\rightarrow\infty}
\int_{\mathbb{R}^{N}}
\int_{\mathbb{R}^{N}}
\frac{|u_{n}(x)|^{2^{*}_{\alpha_{1}}}|u_{n}(y)|^{2^{*}_{\alpha_{1}}}}{|x-y|^{\alpha_{1}}}
\mathrm{d}x
\mathrm{d}y
\right)
^{\frac{s(N-2s)}{N(N+2s-\alpha_{1})}}
=0.
\end{aligned}
\end{equation}
Combining Hardy--Littlewood--Sobolev inequality
and
(\ref{52}),
for all
$i=2,\ldots,k$,
we know
\begin{equation}\label{53}
\begin{aligned}
\lim_{n\rightarrow\infty}
\int_{\mathbb{R}^{N}}
\int_{\mathbb{R}^{N}}
\frac{|u_{n}(x)|^{2^{*}_{\alpha_{i}}}|u_{n}(y)|^{2^{*}_{\alpha_{i}}}}{|x-y|^{\alpha_{i}}}
\mathrm{d}x
\mathrm{d}y
\leqslant&
C
\lim_{n\rightarrow\infty}
\|u_{n}\|_{L^{2^{*}_{s}}(\mathbb{R}^{N})}^{2\cdot2^{*}_{\alpha_{i}}}
=0.
\end{aligned}
\end{equation}
According to
(\ref{51})
--
(\ref{53})
and the definition of
$(PS)_{c^{1}}$
sequence
,
we obtain
$$
c^{1}
+
o(1)
=
\frac{1}{2}
\|u_{n}\|_{D}^{2},$$
and
$$o(1)
=
\|u_{n}\|_{D}^{2},
$$
these imply that
$c^{1}=0$,
which contradicts as $0<c^{1}$.

\noindent
{\bf Case 2.}
From Case 1,
we have
$u_{n}\in \mathcal{E}^{s,\alpha_{1},2^{*}_{\alpha_{1}}}(\mathbb{R}^{N})$.
Applying the result of (ii) in Lemma \ref{lemma6},
we know that
$u_{n}\in \bigcap_{i=2}^{k}
\mathcal{E}^{s,\alpha_{i},2^{*}_{\alpha_{i}}}(\mathbb{R}^{N})$.
Similar to Case 1,
for all
$i=2,\ldots,k$,
we prove that
$$\lim_{n\rightarrow\infty}
\int_{\mathbb{R}^{N}}
\int_{\mathbb{R}^{N}}
\frac{|u_{n}(x)|^{2^{*}_{\alpha_{i}}}|u_{n}(y)|^{2^{*}_{\alpha_{i}}}
}{|x-y|^{\alpha_{i}}}
\mathrm{d}x
\mathrm{d}y>0.$$

\noindent
{\bf Case 3.}
Suppose that
\begin{equation}\label{54}
\begin{aligned}
\lim_{n\rightarrow\infty}
\int_{\mathbb{R}^{N}}
|u_{n}|^{2^{*}_{s}}
\mathrm{d}x
=0,
\end{aligned}
\end{equation}
By using Lemma \ref{lemma1}
and
(\ref{54}),
for all
$i=1,\ldots,k$,
we have
\begin{equation}\label{55}
\begin{aligned}
\lim_{n\rightarrow\infty}
\int_{\mathbb{R}^{N}}
\int_{\mathbb{R}^{N}}
\frac{|u_{n}(x)|^{2^{*}_{\alpha_{i}}}|u_{n}(y)|^{2^{*}_{\alpha_{i}}}}{|x-y|^{\alpha_{i}}}
\mathrm{d}x
\mathrm{d}y
\leqslant&
C
\lim_{n\rightarrow\infty}
\|u_{n}\|_{L^{2^{*}_{s}}(\mathbb{R}^{N})}
^{2\cdot2^{*}_{\alpha_{i}}}=0.
\end{aligned}
\end{equation}
Applying
(\ref{54})
and
(\ref{55}),
we get
$$
c^{1}
+
o(1)
=
\frac{1}{2}
\|u_{n}\|_{D}^{2},$$
and
$$o(1)
=
\|u_{n}\|_{D}^{2},
$$
these imply that
$c^{1}=0$,
which contradicts as $0<c^{1}$.
\qed

We are now ready to prove the existence of nonnegative solution for problem
$(\mathcal{P}_{1})$.

\noindent
{\bf Proof of Theorem \ref{theorem1}:}
\noindent
{\bf Step 1.}
Since
$\{u_{n}\}$
is a bounded sequence in
$D^{s,2}(\mathbb{R}^{N})$,
up to a subsequence,
we can assume that
\begin{align*}
&u_{n}\rightharpoonup u
~
\mathrm{in}
~
D^{s,2}(\mathbb{R}^{N}),~
u_{n}\rightarrow u
~
\mathrm{a.e. ~in}
~
\mathbb{R}^{N},\\
&u_{n}\rightarrow u
~
\mathrm{in}
~
L^{r}_{loc}(\mathbb{R}^{N})
~
\mathrm{for~all}
~
r\in[1,2^{*}_{s}).
\end{align*}
According to
Lemma
\ref{lemma4},
Lemma
\ref{lemma5}
and
Lemma
\ref{lemma19},
there exists
$C>0$
such that
$$
\|u_{n}\|_{\mathcal{L}^{2,N-2s}(\mathbb{R}^{N})}\geqslant C>0.
$$
On the other hand,
since the sequence is bounded in
$D^{s,2}(\mathbb{R}^{N})$
and
$D^{s,2}(\mathbb{R}^{N})\hookrightarrow L^{2^{*}_{s}}(\mathbb{R}^{N})\hookrightarrow \mathcal{L}^{2,N-2s}(\mathbb{R}^{N})$,
we have
$$
\|u_{n}\|_{\mathcal{L}^{2,N-2s}(\mathbb{R}^{N})}\leqslant C,
$$
for some
$C>0$
independent of $n$.
Hence, there exists a positive constant which we denote again by $C$ such
that for any $n$ we obtain
$$
C
\leqslant
\|u_{n}\|_{\mathcal{L}^{2,N-2s}(\mathbb{R}^{N})}\leqslant C^{-1}.
$$
So we may find
$\sigma_{n} > 0$
and
$x_{n}\in \mathbb{R}^{N}$
such that
$$
\frac{1}{\sigma_{n}^{2s}}
\int_{B(x_{n},\sigma_{n})}
|u_{n}(y)|^{2}
\mathrm{d}y
\geqslant
\|u_{n}\|_{\mathcal{L}^{2,N-2s}(\mathbb{R}^{N})}^{2}
-
\frac{C}{2n}
\geqslant
C_{7}>0.
$$
Let
$\bar{v}_{n}(x)=\sigma_{n}^{\frac{N-2s}{2}}u_{n}(x_{n}+\sigma_{n}x)$.
We could verify that
$$I_{1}(\bar{v}_{n})=I(u_{n})\rightarrow c^{1},~
\langle I^{'}_{1}(\bar{v}_{n}),\varphi\rangle
~\mathrm{as}~n\rightarrow\infty.$$
It's obviously that
$\{\bar{v}_{n}\}$
is uniformly bounded in
$D^{s,2}(\mathbb{R}^{N})$.
Thus there exists
$v$
such that
\begin{align*}
&\bar{v}_{n}\rightharpoonup \bar{v}
~
\mathrm{in}
~
D^{s,2}(\mathbb{R}^{N}),~
\bar{v}_{n}\rightarrow \bar{v}
~
\mathrm{a.e. ~in}
~
\mathbb{R}^{N},\\
&\bar{v}_{n}\rightarrow \bar{v}
~
\mathrm{in}
~
L^{r}_{loc}(\mathbb{R}^{N})~~\mathrm{for~all~}r\in[1,2^{*}_{s} ).
\end{align*}
Then
$
\int_{B(0,1)}
|\bar{v}_{n}(y)|^{2}
\mathrm{d}y
\geqslant
C_{7}>0$.
As a result,
$\bar{v}\not\equiv0$.

\noindent
{\bf Step 2.}
Similar to \eqref{39},
we get
\begin{equation*}
\begin{aligned}
\langle I^{'}_{1}(\bar{v}),\varphi\rangle=0.
\end{aligned}
\end{equation*}
Similar to the proof of Theorem \ref{theorem2},
we know that
$\bar{v}_{n}\rightarrow \bar{v}$
strongly in
$D^{s,2}(\mathbb{R}^{N})$,
and
$I_{1}(\bar{v})=c^{1}$.
Moreover,
we can choose
$\bar{v}\geqslant0$.
By using the fractional Kelvin transformation
\begin{equation*}
\begin{aligned}
\bar{\bar{v}}(x)
=
\frac{1}{|x|^{N-2s}}
\bar{v}
\left(
\frac{x}{|x|^{2}}
\right).
\end{aligned}
\end{equation*}
We obtain
\begin{equation*}
\begin{aligned}
(-
\Delta)^{s}
\bar{\bar{v}}
=&
\sum\limits^{k}_{i=1}
\left(
\int_{\mathbb{R}^{N}}
\frac{|\bar{\bar{v}}|^{\frac{2N-\alpha_{i}}{N-2s}}}
{|x-y|^{\alpha_{i}}}
\mathrm{d}y
\right)
\left|
\bar{\bar{v}}
\right|^{\frac{4s-\alpha_{i}}{N-2s}}
\bar{\bar{v}}
+
\left|
\bar{\bar{v}}
\right|^{\frac{4s}{N-2s}}
\bar{\bar{v}}
,
~\mathrm{in}~\mathbb{R}^{N}\backslash\{0\}.
\end{aligned}
\end{equation*}
\qed
\section*{Open Problem}
During the preparation of the manuscript we faced one problem which is worth to be tackled in forthcoming investigation.

We want to generalize the study of problem $(\mathcal{P}_{1})$ to the following problem:
$$
(-\Delta)^{s} u
-\zeta\frac{ u}{|x|^{2s}}
=
\sum_{i=1}^{k}
\left(
\int_{\mathbb{R}^{N}}
\frac{|u|^{2^{*}_{\alpha_{i}}}}{|x-y|^{\alpha_{i}}}
\mathrm{d}y
\right)
|u|^{2^{*}_{\alpha_{i}}-2}u
+
|u|^{2^{*}_{s}-2}u
,
\mathrm{~in~}
\mathbb{R}^{N},
\eqno(\mathcal{P}_{4})
$$
where
$N\geqslant3$,
$s\in(0,1)$,
$\zeta\in
\left(
0,4^{s}\frac{\Gamma(\frac{N+2s}{4})}{\Gamma(\frac{N-2s}{4})}
\right)$
and
$0<\alpha_{1}<\alpha_{2}<\cdots<\alpha_{k}<N$
($k\in \mathbb{N}$,
$2\leqslant k<\infty$).
If
$\zeta=0$,
then problem $(\mathcal{P}_{4})$
goes back to
problem $(\mathcal{P}_{1})$.
However,
they are very different from each other.

\end{document}